\documentclass[12pt]{amsart}

\usepackage{amssymb,amsbsy}
\usepackage{latexsym,array}
\usepackage{verbatim,color}
\usepackage{fullpage,euscript}
\usepackage{xcolor}
\usepackage{tikz}
\usetikzlibrary{arrows,shapes,trees}
\usepackage[colorlinks=true,linkcolor=blue,urlcolor=my_color,citecolor=magenta]{hyperref}

\definecolor{my_color}{rgb}{0,0.5,0.5}
\definecolor{MIXT}{rgb}{0.8,0.5,0.2}
\definecolor{mixt}{rgb}{0.5,0.3,0.2}
\definecolor{sin}{rgb}{0,0.5,0.5}
\definecolor{darkblue}{rgb}{0,0.1,0.8}
\definecolor{redi}{rgb}{0.5,0,0.4}

\numberwithin{equation}{section}

\input {cyracc.def}
\font\tencyr=wncyr10 
\font\tencyi=wncyi10 
\font\tencysc=wncysc10 
\def\rus{\tencyr\cyracc}

\def\rusi{\tencyi\cyracc}
\def\rusc{\tencysc\cyracc}

\newtheorem{thm}{Theorem}[section]
\newtheorem{lm}[thm]{Lemma}
\newtheorem{cl}[thm]{Corollary}
\newtheorem{prop}[thm]{Proposition}
\newtheorem{conj}[thm]{Conjecture}
\newtheorem{qtn}[thm]{Question}

\theoremstyle{remark}
\newtheorem{rmk}[thm]{Remark}

\theoremstyle{definition}
\newtheorem{ex}[thm]{Example} 
\newtheorem{df}{Definition}

\newtheorem*{rema}{Remark}

\newcommand {\ce}{{\mathfrak c}}

\newcommand {\g}{{\mathfrak g}}

\newcommand {\h}{{\mathfrak h}}

\newcommand {\el}{{\mathfrak l}}

\newcommand {\q}{{\mathfrak q}}

\newcommand {\es}{{\mathfrak s}}
\newcommand {\te}{{\mathfrak t}}

\newcommand {\z}{{\mathfrak z}}

\newcommand{\gt}{\mathfrak}

\newcommand {\eus}{\EuScript}

\newcommand {\gA}{{\eus A}}

\newcommand {\gS}{{\eus S}}
\newcommand {\gZ}{{\eus Z}}

\newcommand {\esi}{\varepsilon}

\newcommand {\ap}{\alpha}

\newcommand {\lb}{\lambda}
\newcommand {\vp}{\varphi}

\newcommand {\ca}{{\mathcal A}}

\newcommand {\gc}{{\mathcal C}}

\newcommand {\cz}{{\mathcal Z}}

\newcommand {\BP}{{\mathbb P}}
\newcommand {\BV}{{\mathbb V}}

\newcommand {\BZ}{{\mathbb Z}}

\newcommand {\md}{/\!\!/}

\newcommand{\id}{{\mathsf{id}}}
\newcommand {\ad}{{\mathrm{ad\,}}}

\newcommand {\codim}{{\mathrm{codim\,}}}

\newcommand {\ind}{{\mathrm{ind\,}}}
\newcommand {\Lie}{{\mathsf{Lie\,}}}

\newcommand {\rk}{{\mathsf{rk\,}}}
\newcommand {\crk}{{\mathsf{rkc}}}

\newcommand {\trdeg}{{\mathrm{tr.deg\,}}}

\newcommand {\tri}{\mathfrak{sl}_2}

\newcommand {\un}{\underline}

\newcommand {\bb}{{\boldsymbol{b}}}
\newcommand {\bx}{\boldsymbol{\zeta}}

\newcommand {\cW}{{\mathcal W}}
\newcommand {\An}{{\mathbb A^n}}
\newcommand {\PC}{Poisson commutative}

\newcommand {\beq}{\begin{equation}}
\newcommand {\eeq}{\end{equation}}

\renewcommand{\le}{\leqslant}
\renewcommand{\ge}{\geqslant}

\newcommand {\bbk}{\Bbbk}

\begin{document}
\setlength{\parskip}{3pt plus 2pt minus 0pt}
\hfill { {\color{blue}\scriptsize September 2, 2018}}
\vskip1ex

\title[Poisson-commutative subalgebras]
{Poisson-commutative subalgebras  of $\gS(\g)$ associated with involutions} 
\author[D.\,Panyushev]{Dmitri I. Panyushev}
\address[D.\,Panyushev]%
{Institute for Information Transmission Problems of the R.A.S, Bolshoi Karetnyi per. 19, 
Moscow 127051, Russia}
\email{panyushev@iitp.ru}
\author[O.\,Yakimova]{Oksana S.~Yakimova}
\address[O.\,Yakimova]{Universit\"at zu K\"oln,
Mathematisches Institut, Weyertal 86-90, 50931 K\"oln, Deutschland}
\email{yakimova.oksana@uni-koeln.de}
\thanks{The research of the first author is partially supported by the R.F.B.R. grant {\rus N0} 16-01-00818. 
The second author is supported by  the DFG (Heisenberg--Stipendium).}
\maketitle
\begin{abstract}
The symmetric algebra $\gS(\g)$ of a reductive Lie algebra $\g$ is equipped with the standard 
Poisson structure, i.e., the Lie--Poisson bracket.  
Poisson-commutative subalgebras of  $\gS(\g)$ attract a great deal of attention, because of
their relationship to integrable systems and, more recently, to geometric representation theory. 
The transcendence degree of  a Poisson-commutative subalgebra  ${\mathcal C}\subset\gS(\g)$
is bounded by the ``magic number'' $\bb(\g)$ of $\g$. 
The ``argument shift method'' of Mishchenko--Fomenko was basically the only known source of 
$\mathcal C$ with $\trdeg{\mathcal C}=\bb(\g)$. We introduce an essentially different construction related 
to symmetric decompositions $\g=\g_0\oplus\g_1$. Poisson-commutative subalgebras 
$\cz,\tilde\cz\subset\gS(\g)^{\g_0}$ of the maximal possible transcendence degree are presented. If the $\BZ_2$-contraction $\g_0\ltimes\g_1^{\sf ab}$ has a polynomial ring of symmetric invariants, then
$\tilde\cz$ is  
a polynomial maximal  Poisson-commutative subalgebra of $\gS(\g)^{\g_0}$, and its free generators
are explicitly described. 
\end{abstract}

\tableofcontents

\section*{Introduction}

\noindent
The ground field $\bbk$ is algebraically closed and of characteristic $0$.   
A commutative associative $\bbk$-algebra ${\ca}$  is a {\it Poisson algebra\/} if 
there is   an additional anticommutative bilinear operation
$\{\,\,,\,\}\!:\,{\ca}\times{\ca}\to{\ca}$ called a {\it Poisson bracket} such that  
\[
\begin{array}{cl}
\{a,bc\}=\{a,b\}c+b\{a,c\}, &  \text{(the Leibniz rule)} \\
\{a,\{b,c\}\}+\{b,\{c,a\}\}+\{c,\{a,b\}\}=0 & \text{(the Jacobi identity)}
\end{array}
\]
for all $a,b,c\in{\ca}$. A subalgebra ${\mathcal C}\subset{\ca}$ is 
{\it Poisson-commutative} if $\{\mathcal{C},\mathcal{C}\}=0$.  
The {\it Poisson centre} $\mathcal{ZA}$ of ${\ca}$ is defined by the condition 
$\mathcal{ZA}=\{z\in\ca \mid \{z,a\}=0  \ \forall a\in{\ca}\}$. 

Usually, Poisson algebras occur as algebras of functions on varieties (manifolds), and 
we are only interested in the case, where such a variety is an affine $n$-space $\An$ and hence
$\ca=\bbk[\An]$ is a polynomial ring in $n$ variables. Two Poisson brackets on $\An$ are said to be 
{\it compatible}, if all their linear combinations are again Poisson brackets. 
 
There is a general method for constructing a ``large'' Poisson-commutative subalgebra of $\ca$ 
associated with a pair of compatible brackets, see e.g.~\cite{bol-bor}. Let $\{\ ,\ \}'$ and $\{\ ,\ \}''$ be 
compatible Poisson brackets on $\An$.
This yields a two parameter family of Poisson brackets $a\{\ ,\ \}'+b\{\ ,\ \}''$, $a,b\in\bbk$. As we are only 
interested in the corresponding Poisson centres, it is convenient to organise this, up to scaling, in a 
1-parameter family $\{\ ,\ \}_t=\{\ ,\ \}'+t\{\ ,\ \}''$, $t\in\BP=\bbk\cup\{\infty\}$, where $t=\infty$ 
corresponds to the bracket $\{\ ,\ \}''$.  
The {\it central  rank\/} $\crk\{\,\,,\,\}$ of a Poisson bracket $\{\,\,,\,\}$ is defined as the codimension of a symplectic leaf in general position, see Definition~\ref{def-crk}. 
For almost
all $t\in\BP$, $\crk\{\,\,,\,\}_t$ has one and the same (minimal) value, and we set
$\BP_{\sf reg}=\{t\in \BP\mid \crk\{\,\,,\,\}_t \text{ is minimal}\}$, $\BP_{\sf sing}=\BP\setminus \BP_{\sf reg}$. 
Let $\cz_t$ denote the centre of $(\ca,\{\ ,\ \}_t)$. 
The key fact is that the algebra
$\gZ$ generated by $\{\cz_t\mid t\in\BP_{\sf reg}\}$ is Poisson-commutative w.r.t{.} to any bracket in the family. 
In many cases, this construction provides a Poisson-commutative subalgebra of $\ca$ of maximal transcendence degree.
We demonstrate this with a well-known important example.

\begin{ex}          \label{ex:q-compatible}
For any finite-dimensional Lie algebra $\q$, the dual space $\gt q^*$ has a Poisson structure.  
Here $\bbk[\gt q^*]\cong \gS(\gt q)$ and the Lie--Poisson bracket $\{\ ,\ \}_{\sf LP}$ is defined by
$\{\xi,\eta\}_{\sf LP}=[\xi,\eta]$ for $\xi,\eta\in\gt q$. 
The Poisson centre of $(\gS(\gt q), \{\ ,\ \}_{\sf LP})$ coincides with the ring $\gS(\gt q)^{\gt q}$ of 
symmetric $\gt q$-invariants. The celebrated ``argument shift method'', which goes back to 
Mishchenko--Fomenko \cite{mf}, provides large Poisson-commutative subalgebras of $\gS(\q)$ 
starting from the Poisson centre $\gS(\q)^\q$. Given $\gamma\in\q^*$, the $\gamma$-shift of argument 
produces the {\it Mishchenko--Fomenko subalgebra\/} ${\ca}_\gamma$. Namely, for 
$F\in\gS(\q)=\bbk[\q^*]$, let $\partial_{\gamma} F$ be the directional derivative of $F$ 
with respect to $\gamma$, i.e., 
\[
    \partial_{\gamma}F(x)=\frac{\textsl{d}}{\textsl{d}t} F(x+t\gamma)\Big|_{t=0}.
\]
Then ${\ca}_\gamma$ is generated by all $\partial_\gamma^k F$ with $k\ge 0$ for all  
$F\in\gS(\q)^{\q}$. The core of this method is that for any $\gamma\in\gt q^*$ 
there is the Poisson bracket $\{\,\,,\,\}_{\gamma}$ on $\q^*$ such that 
$\{\xi,\eta\}_{\gamma}=\gamma([\xi,\eta])$ for $\xi,\eta\in\q$, and that this new 
bracket is compatible with $\{\ ,\ \}_{\sf LP}$.
One can prove that $\crk\{\,\,,\,\}_t$ takes one and the same value  
for all $\{\ ,\ \}_t=\{\ ,\ \}_{\sf LP}+t\{\,\,,\,\}_{\gamma}$ with $t\in\bbk$, i.e., $\bbk\subset\BP_{\sf reg}$, and
${\ca}_\gamma$ is generated by all the corresponding  centres $\cz_t$, $t\in\bbk$. (Actually, $\BP_{\sf reg}=\BP$ if and only if
$\gamma$ is regular in $\gt q^*$.)
The importance of these subalgebras and their quantum counterparts 
is explained e.g. in~\cite{FFR,v:sc}. If $\q$ is reductive and $\gamma$ is regular, then
${\ca}_\gamma$ is a maximal Poisson-commutative subalgebra of $\gS(\q)$~\cite{mrl}.
\end{ex}

Our main object is a certain 1-parameter family of Poisson brackets on the dual of a semisimple Lie algebra 
$\g$. Let $\sigma$ be an involution of $\g$ and $\g=\g_0 \oplus \g_1$ the corresponding 
$\BZ_2$-{\it grading\/} (or {\it symmetric decomposition}). We also say that $(\g,\g_0)$ is a {\it symmetric 
pair}. Without loss of generality, we may assume that the pair $(\g,\sigma)$ is 
{\it indecomposable}, i.e., $\g$ has no proper 
$\sigma$-stable ideals. Then either $\g$ is simple or $\g=\h\oplus\h$, where $\h$ is simple and 
$\sigma$ is a permutation. Our family of Poisson brackets is related to the decomposition:
\[
   \{\,\,,\,\}_{\sf LP}=\{\,\,,\, \}_{0,0}+\{\,\,,\,\}_{0,1}+\{\,\,,\,\}_{1,1},
\]
where $\{\,\,,\,\}_{i,j}=[\,\,,\,]_{i,j}\!: \g_i\times\g_j \to \g_{i+ j}$ for $i,j\in\BZ_2\simeq \{0,1\}$, see 
Section~\ref{sect:2} for details. Using this, we consider the $1$-parameter family of Poisson brackets 
on $\g^*$:
\beq     \label{eq:family}
   \{\,\,,\,\}_{t}=\{\,\,,\, \}_{\sf 0,0}+\{\,\,,\,\}_{\sf 0,1}+t\{\,\,,\,\}_{\sf 1,1},
\eeq
where $t\in \BP$ and $\{\,\,,\,\}_{\infty}=\{\,\,,\,\}_{\sf 1,1}$.
Each element of this family is a Poisson bracket and here $\BP_{\sf reg}=\bbk$ unless $\g=\tri$. 
For $\tri$, one has $\BP_{\sf reg}=\BP$, and this case has to be considered separately. Nevertheless, 
the final result can be stated uniformly, for all simple $\g$, see below.

Let $\cz_t$ ($t\in\BP$) denote the centre of $(\gS(\g), \{\,\,,\,\}_{t})$ and
$\gZ$  the subalgebra of $\gS(\g)$ generated by all $\cz_t$  
with $t\in \BP_{\sf reg}$. Then $\{\gZ,\gZ\}_{\sf LP}=0$. Moreover, 
$\{\g_0,\gZ\}_{\sf LP}=0$, i.e., $\gZ$ is a \PC\ subalgebra of $\gS(\g)^{\g_0}$.
By~\cite[Prop.\,1.1]{m-y}, we have
\[
   \trdeg{\mathcal C}\le \frac{1}{2}(\dim\g_1+\rk\g+\rk\g_0)
\]
for any Poisson-commutative subalgebra ${\mathcal C}\subset\gS(\g)^{\g_0}$. 
We prove that this upper bound is attained for $\gZ$, see Theorem~\ref{prop:main2}.

The computation of $\trdeg \gZ$ is completely general and is valid for any $\sigma$. However, this is not 
the case with more subtle properties. Our goal is to realise whether $\gZ$ is polynomial and is 
maximal \PC\ in $\gS(\g)^{\g_0}$. For $t=0$ in Eq.~\eqref{eq:family}, one obtains the Lie--Poisson
bracket of the Lie algebra $\g_{(0)}:=\g_0\ltimes \g_1^{\sf ab}$. The symmetric invariants of $\g_{(0)}$
have intensively been studied in \cite{coadj, contr,Y-imrn}. The output is that there are four ``bad'' 
involutions of a simple $\g$ in which $\gS(\g_{(0)})^{\g_{(0)}}$ is not polynomial. These cases are 
related to $\g$ of type $\eus E_n$. In all other cases, $\gS(\g)^\g$ has a {\it good generating system} 
(=\,{\sf g.g.s.}) for $(\g,\g_0)$, say $H_1,\dots,H_l$ ($l=\rk\g$), and a set of free 
generators of $\gS(\g_{(0)})^{\g_{(0)}}$ is then obtained from the $H_i$'s via a simple procedure, see 
Section~\ref{sect:3} for details.
\\ \indent
In the rest of the introduction, we assume that $\sigma$ is ``good'' and $\g\ne\tri$. In particular, there 
is a {{\sf g.g.s.}} for $(\g,\g_0)$. 
This is of vital importance for us, because we then prove that
$\gZ$ is freely generated by the nonzero bi-homogeneous components of all
$H_i$'s and is therefore polynomial, see Theorems~\ref{thm:main3-2} and~\ref{free-main}.  
Let $r_0\!: \gS(\g)^{\g}\to \gS(\g_0)^{\g_0}$ be the restriction homomorphism
related to the embedding $\g^*_0\hookrightarrow \g^*= \g^*_0\oplus \g^*_1$.
Furthermore,
\\ \indent
\textbullet \quad $\gZ$ is a maximal \PC\ subalgebra of $\gS(\g)^{\g_0}$ if and only if $r_0$
is onto, see Theorem~\ref{thm:maxim1}.
\\ \indent
 \textbullet \quad  In general, let $\tilde\gZ$ be the subalgebra of $\gS(\g)$ generated by 
$\gZ$ and $\gS(\g_0)^{\g_0}$. (Hence $\tilde\gZ=\gZ$ if and only if $r_0$ is onto.)
We prove that $\tilde\gZ$ is still polynomial  and that 
it is a maximal \PC\ subalgebra of $\gS(\g)^{\g_0}$, see Theorem~\ref{thm:maxim2}. 
This statement also embraces the $\tri$-case, because then $\gZ=\tilde \gZ$ is polynomial, etc.

In Section~\ref{fancy}, we present a Poisson interpretation of the Kostant regularity criterion for 
$\g$~\cite[Theorem~9]{ko63} and give new related formulas arising from $\BZ_2$-gradings and compatible Poisson structures. As a by-product, we describe $\cz_\infty$ for all $\sigma$. 

Section~\ref{sec-q} contains a discussion of possible quantisations of $\gZ$ and $\tilde\gZ$,
i.e., their lifting to the enveloping algebra $\eus U(\g)$. We conjecture that quantum analogues of these algebras may have applications in representation theory, and more explicitly, in the branching problem 
$\g\downarrow\g_0$.  In Section~\ref{sec-cl}, it is explained how to construct a 
polynomial maximal  Poisson-commutative subalgebra of $\gS(\g)$ related to 
a chain of symmetric subalgebras
\[
\g=\g^{(0)}\supset \g^{(1)}\supset \g^{(2)} \supset\ldots \supset \g^{(m)}
\]
with $[\g^{(m)},\g^{(m)}]=0$. 

In the Appendix, we gather auxiliary results on the kernels of a 1-parameter family of skew-symmetric bilinear forms on a vector space.

We refer to \cite{duzu} for generalities on Poisson varieties, Poisson tensors, symplectic leaves, etc.

\section{Preliminaries on the coadjoint representation}
\label{sect:prelim}

\noindent
Let $Q$\/ be a connected affine algebraic group with Lie algebra $\q$. The symmetric algebra 
$\gS(\q)$ over $\bbk$ is identified with the graded algebra of polynomial functions on $\q^*$, and we also 
write $\bbk[\q^*]$ for it.  
\\ \indent
Let $\q^\xi$ denote the stabiliser in $\q$ of $\xi\in\q^*$. The {\it index of}\/ $\q$, $\ind\q$, is the minimal codimension of $Q$-orbits in $\q^*$. Equivalently,
$\ind\q=\min_{\xi\in\q^*} \dim \q^\xi$. By Rosenlicht's theorem~\cite[2.3]{VP}, one also has
$\ind\q=\trdeg\bbk(\q^*)^Q$. The ``magic number'' associated with $\q$ is $\bb(\q)=(\dim\q+\ind\q)/2$. 
Since the coadjoint orbits are even-dimensional, the magic number is an integer. If $\q$ is reductive, then  
$\ind\q=\rk\q$ and $\bb(\q)$ equals the dimension of a Borel subalgebra. The Lie--Poisson bracket on
$\bbk[\q^*]$ is defined on the elements of degree $1$ (i.e., on $\q$) by $\{x,y\}_{\sf LP}:=[x,y]$. 
The {\it Poisson centre\/} of $\gS(\q)$ is 
\[
    \gS(\q)^\q=\{H\in \gS(\q)\mid \{H,x\}_{\sf LP}=0 \ \ \forall x\in\q\} .
\] 
As $Q$ is connected, we have $\gS(\q)^\q=\gS(\q)^{Q}=\bbk[\q^*]^Q$.
The set of $Q$-{\it regular\/} elements of $\q^*$ is 
\beq       \label{eq:regul-set}
    \q^*_{\sf reg}=\{\eta\in\q^*\mid \dim \q^\eta=\ind\q\} .
\eeq
Set $\q^*_{\sf sing}=\q^*\setminus \q^*_{\sf reg}$.
We say that $\q$ has the {\sl codim}--$n$ property if $\codim \q^*_{\sf sing}\ge n$.  By~\cite{ko63}, the semisimple algebras $\g$ have the {\sl codim}--$3$ property.

\subsection{The Poisson tensor}    
\label{subs:P-tensor}
Let $\Omega^i$ be the $\ca$-module of differential $i$-forms on $\An$. Then
$\Omega=\bigoplus_{i=0}^n \Omega^i$ is the $\ca$-algebra of regular   
differential forms on $\An$. Likewise,
$\cW=\bigoplus_{i=0}^n \cW^i$  is the graded skew-symmetric algebra  
of polyvector fields generated by the $\ca$-module $\cW^1$ of polynomial vector fields on $\An$.
Both algebras are free $\ca$-modules. 
If $\An$ has a Poisson structure $\{\ ,\ \}$, then $\pi$ is the corresponding {\it Poisson tensor (bivector)}.
That is,  $\pi\in \operatorname{Hom}_{\ca}(\Omega^2,{\ca})$ is defined by the equality
$\pi(\textsl{d}f\wedge \textsl{d}g)=\{f,g\}$ for $f,g\in{\ca}$. 
Then $\pi(x)$, $x\in\An$, defines a skew-symmetric bilinear form on $T^*_x(\An)\simeq
(\An)^*$.  Formally, if $v=\textsl{d}_xf$ and $u=\textsl{d}_xg$ for $f,g\in\ca$, then
$\pi(x)(v,u)=\pi(\textsl{d}f\wedge \textsl{d}g)(x)=\{f,g\}(x)$.

\begin{df} \label{def-crk}
The {\it central rank\/} of a Poisson bracket $\{\,\,,\,\}$ on $\An$, denoted $\crk\{\,\,,\,\}$, is the 
minimal codimension of the symplectic leaves in $\An$.   
\end{df}
\vskip-1ex
It is easily seen that if $\pi$ is the corresponding Poisson tensor, then \\
\centerline{$\crk\{\,\,,\,\}=\min_{x\in\An} \dim  \ker \pi(x)=n-\max_{x\in\An}\rk \pi(x)$.}

{\bf Example.} For a Lie algebra $\q$ and the dual space $\q^*$ equipped with the Lie--Poisson bracket 
$\{\,\,,\,\}_{\sf LP}$, the symplectic leaves are the coadjoint $Q$-orbits. Hence 
$\crk\{\,\,,\,\}_{\sf LP}=\ind\q$.

In view of the duality between differential 1-forms and vector 
fields, we may regard $\pi$ as an element of $\cW^2$. Let $[[\ ,\ ]]: \cW^i\times \cW^j \to \cW^{i+j-1}$
be the Schouten bracket.
The Jacobi identity for $\pi$ is equivalent to that 
$[[\pi,\pi]]=0$, see e.g.~\cite[Chapter\,1.8]{duzu}. 

\begin{lm}     \label{lm:compat}
Two Poisson brackets $\{\ ,\ \}'$ and $\{\ ,\ \}''$ are compatible if and only if a sole
linear combination, non-proportional to either of the initial brackets, is a Poisson bracket. 
\end{lm}
\begin{proof}
In place of  Poisson brackets, we may consider the corresponding Poisson tensors. Given two tensors
$\pi'$ and $\pi''$, consider $R=a\pi'+b\pi''$ with $a,b\in \bbk^\times$. Then  $R$ is a Poisson tensor if and only if $[[R,R]]=0$. In view of the fact that $[[\pi',\pi'']]=[[\pi'',\pi']]$, this reduces to the condition 
$[[\pi',\pi'']]=0$ regardless of nonzero $a,b$.
\end{proof}

\subsection{Contractions and compatibility}     \label{subs:contr-compat}
Let  $\q=\h \oplus V$ be a vector space decomposition, 
where $\h$ is a subalgebra. For any $s\in\bbk^{\times}$, define a linear map 
$\varphi_s\!:\gt q\to\gt q$ by 
setting $\varphi_s\vert_{\h}=\id$, $\varphi_s\vert_{V}=s{\cdot}\id$. Then $\vp_s\vp_{s'}=\vp_{ss'}$ and
$\varphi_s^{-1}=\varphi_{s^{-1}}$, i.e., this yields a one-parameter subgroup of ${\rm GL}(\gt q)$. The 
invertible map $\varphi_s$ defines a new (isomorphic to the initial) Lie algebra structure $[\,\,,\,]_{(s)}$ 
on the same vector space $\q$ by the formula 
\beq     \label{eq:fi_s}
    [x,y]_{(s)}=\vp_s^{-1}([\vp_s(x),\varphi_s(y)]).
\eeq
The corresponding Poisson bracket is $\{\ ,\ \}_{s}$. We naturally extend $\vp_s$ to an automorphism of 
$\gS(\q)$.  Then the centre of the Poisson algebra $(\gS(\q),\{\,\,,\,\}_s)$ equals
$\vp_s^{-1}(\gS(\q)^{\q})$. 

The condition $[\h,\h]\subset \h$ implies that there is the limit of the brackets $[\ ,\ ]_{(s)}$
as $s$ tends to zero. The limit bracket is denoted by $[\ ,\ ]_{(0)}$ and the corresponding Lie algebra 
is the semi-direct product $\h{\ltimes}V^{\sf ab}$, where $[V^{\sf ab},V^{\sf ab}]_{(0)}=0$.
The algebra $\h \ltimes V^{\sf ab}$ is called an {\it In\"on\"u-Wigner} or 
{\it one-parameter contraction\/} of $\q$, see e.g. \cite{alafe,contr}.

Having a family of Poisson brackets $\{\,\,,\,\}_s$ on $\q^*$ associated with the 
maps $\varphi_s$, it is natural to ask whether these brackets are compatible. 

\begin{lm}           \label{lm:comp-bra}    
As above, let $\q=\h\oplus V$, where $\h\subset\q$ is a subalgebra. Let $s,s'\in\bbk$.
\begin{itemize}
\item[\sf (i)] \ If $(\q,\h)$ is a symmetric pair, 
i.e., $[\h,V]\subset V$ and $[V,V]\subset\h$, then $\{\,\,,\,\}_s=\{\,\,,\,\}_{-s}$ and 
$\{\,\,,\,\}_s+\{\,\,,\,\}_{s'}=2\{\,\,,\,\}_{\tilde s}$ with $2\tilde s^2=s^2+(s')^2$. 
\item[\sf (ii)] \ If\/ $[V,V]\subset V$, i.e., $V$ is a subalgebra, too, then $\{\,\,,\,\}_s+\{\,\,,\,\}_{s'}=2\{\,\,,\,\}_{\tilde s}$ with $2\tilde s=s+ s'$. 
\end{itemize}
\end{lm}
\begin{proof}
All statements are verified by an easy direct computation. 
\end{proof}

In this article, we are interested in case {\sf (i)} of Lemma~\ref{lm:comp-bra} 
under the assumption that $\q$ is semisimple.

\section{Constructing a Poisson-commutative subalgebra $\gZ$}
\label{sect:2}

\noindent
Let $\g$ be a $\BZ_2$-graded semisimple Lie algebra and $\sigma$ the corresponding
involution of $\g$, i.e., $\g=\g_0\oplus\g_1$ and $\sigma(x)=(-1)^j x$ for $x\in \g_j$. Occasionally, we will 
need the related connected algebraic groups $G$ and $G_0$, i.e., $\g=\Lie(G)$ and $\g_0=\Lie(G_0)$. 
We may assume that $G_0\subset G$.
Under the presence of $\sigma$, the Lie--Poisson bracket is being decomposed as follows:
\[
   \{\,\,,\,\}_{\sf LP}=\{\,\,,\,\}_{\sf 0,0}+\{\,\,,\,\}_{\sf 0,1}+\{\,\,,\,\}_{\sf 1,1}.
\]
More precisely, if $x=x_0+x_1\in \g$, then $\{x,y\}_{0,0}=[x_0,y_0]$,
$\{x,y\}_{0,1}=[x_0,y_1]+[x_1,y_0]$, and $\{x,y\}_{1,1}=[x_1,y_1]$.
Using this decomposition, we introduce a $1$-parameter family of Poisson brackets on $\g^*$:
\[
   \{\,\,,\,\}_{t}=\{\,\,,\,\}_{\sf 0,0}+\{\,\,,\,\}_{\sf 0,1}+t\{\,\,,\,\}_{\sf 1,1},
\] 
where $t\in \BP=\bbk\cup\{\infty\}$ and $\{\,\,,\,\}_{\infty}=\{\,\,,\,\}_{\sf 1,1}$.
It is easily seen that  $\{\,\,,\,\}_{t}$ with $t\in \bbk^\times$ is given by the map $\vp_s$, where $s^2=t$
(see Section~\ref{subs:contr-compat}), and it follows from Lemmas~\ref{lm:compat} and 
\ref{lm:comp-bra} that all these brackets are compatible. Hence 
\[
    \{\,\,,\,\}_{t}=\{\,\,,\,\}_{0}+t\{\,\,,\,\}_{\infty}, \quad t\in\BP, 
\]
in accordance with the general method outlined in the introduction,
Note that $\{\,\,,\,\}_{{\sf LP}}=\{\,\,,\,\}_{0}+\{\,\,,\,\}_{\infty}$. Write $\g_{(t)}$ for the Lie algebra 
corresponding to $ \{\,\,,\,\}_{t}$. Of course, we merely write $\g$ in place of $\g_{(1)}$.
All Lie algebras $\g_{(t)}$ have the same underlying vector space $\g$. 
\\ \indent
{\bf Convention.} 
We identify $\g$,\,$\g_0$, and $\g_1$ with their duals via the Killing form on $\g$. Hence
$\g^*_0\oplus\g^*_1\simeq \g_0\oplus\g_1$. We regard $\g^*$ as the dual of any algebra $\g_{(t)}$
and sometimes omit the subscript `$(t)$' in $\g_{(t)}^*$.  
However, if $\xi\in\g^*$,
then the stabiliser of $\xi$ in the Lie algebra $\g_{(t)}$ (i.e., with respect to the coadjoint representation
of $\g_{(t)}$) is denoted by $\g_{(t)}^\xi$.

Let $\pi_t$ be the Poisson tensor for $\{\,\,,\,\}_{t}$ and
$\pi_t(\xi)$  the skew-symmetric bilinear form on $\g\simeq T^*_\xi(\g^*)$ corresponding to $\xi\in\g^*$,
cf. Section~\ref{subs:P-tensor}. A down-to-earth description is that
$\pi_t(\xi)(x_1,x_2)=\{x_1,x_2\}_{(t)}(\xi)$. Set $\rk\pi_t=\max_{\xi\in\g^*}\rk\pi_t(\xi)$.

\begin{lm}       \label{lm:rank-ot-t}
We have \ $\ind \g_{(t)}=\crk \{\,\,,\,\}_{t}=\begin{cases}  \rk\g, &  t\ne\infty; \\
\dim\g_0+\rk\g-\rk\g_0, & t=\infty. \end{cases}$    
\end{lm}
\begin{proof}  We know that $\crk\{\,\,,\,\}_{\sf LP}=\crk\{\,\,,\,\}_{1}=\rk\g$, if $\g$ is semisimple.

1) \ If $t\ne 0,\infty$, then 
the existence of $\varphi_s$ with $s^2=t$ implies that $\{\,\,,\,\}_{t}$ is isomorphic to 
$\{\,\,,\,\}_{1}$. For $t=0$, one obtains the Poisson bracket of the semi-direct 
product ($\BZ_2$-contraction) $\g_{(0)}=\g_0\ltimes\g_1^{\sf ab}$, and it is proved in \cite[Cor.\,9.4]{p05} that 
$\ind(\g_0\ltimes\g_1^{\sf ab})=\rk\g$.

2) \  By definition, $\crk\{\,\,,\,\}_{\infty}=\ind \g_{(\infty)}=\min_{\xi\in \g^*} \dim\g^\xi_{(\infty)}$.
Here  $\{\,\,,\,\}_{\infty}$ represents the degenerated Lie algebra structure on 
the vector space $\g$ such that $[x_0+x_1,y_0+y_1]_\infty=[x_1,y_1]\in \g_0$.  
One easily verifies that if $\xi=\xi_0+\xi_1\in \g^*$, then
 $\g^\xi_{(\infty)}=\g_0\oplus \g_1^{\xi_0}$. Therefore, 
\[
   \ind \g_{(\infty)}=\dim\g_0+\min_{\xi_0\in \g^*_0} \dim \g_1^{\xi_0}=
   \dim\g- \max_{\xi_0\in \g_0} \dim [\g_1, {\xi_0}] .
\]
In the last step, we use the fact that upon the identification of $\g_0^*$ and $\g_0$, the coadjoint action of
$\g_1\subset \g_{(\infty)}$ on $\g^*_0\subset \g^*_{(\infty)}$   becomes the usual bracket in $\g$.
\\ \indent
By a well-known property of $\BZ_2$-gradings, $\g_0$ always contains a regular semisimple element of
$\g$. If  $\xi_0\in\g_0$ is regular semisimple in $\g$ and hence in $\g_0$, then 
$[\g,\xi_0]=[\g_0,\xi_0]\oplus [\g_1,\xi_0]$, 
$\dim [\g,\xi_0]=\dim\g-\rk\g$, and $\dim[\g_0,\xi_0]=\dim\g_0-\rk\g_0$. 
Hence 
\[
\max_{\xi_0\in \g_0} \dim [\g_1, {\xi_0}]=\dim\g_1+\rk\g_0-\rk\g ,
\]
and we are done.
\end{proof}
It follows from Lemma~\ref{lm:rank-ot-t} that $t=\infty$ is regular in $\BP$ if and only if
$\dim\g_0=\rk\g_0$, i.e., $\g_0$ is Abelian. For the indecomposable pairs, this happens if and only if
$\g=\tri$. For this reason, it is necessary to handle the $\tri$-case separately.

\begin{ex}    \label{ex:sl2}
Let $\g=\tri$ with a standard basis $\{e,h,f\}$ such that $[h,e]=2e,\ [h,f]=-2f, \ [e,f]=h$. Then 
$\gS(\tri)^{\tri}=\bbk[h^2+4ef]$. For the unique (up to conjugation)  non-trivial $\sigma$, one has 
$\g_0=\bbk h$ and $e,f\in\g_1$. Then $\cz_t$ ($t\ne 0,\infty$) is generated by 
$h^2+t^{-1}ef$. An easy calculation shows that $\cz_0=\bbk[ef]$ and $\cz_\infty=\bbk[h]$. 
Here $\BP_{\sf reg}=\BP$, hence   $\gZ$ is generated by all
$\cz_t$ with $t\in \BP$ and  $\gZ=\bbk[h,ef]$. This is a maximal Poisson-commutative
subalgebra of $\gS(\g)$ and it  lies in $\gS(\g)^{\g_0}$.
\end{ex}

Unless otherwise explicitly stated, we assume below that $\g\ne\tri$.
We then obtain a $1$-parameter family of compatible Poisson brackets on $\g^*$, with generic central 
rank being equal to $\rk\g$ and  $\BP_{\sf sing}=\{\infty\}$, where the central rank jumps up to 
$\dim\g_0+\rk\g-\rk\g_0$. Hence $\BP_{\sf reg}=\BP\setminus \{\infty\}=\bbk$.  For each Lie algebra
$\g_{(t)}$, there is the related singular set $\g^*_{(t),\sf sing}=\g^*\setminus \g^*_{(t),\sf reg}$, 
cf.~Eq.~\eqref{eq:regul-set}. Then, clearly,
\[
        \g^*_{(t),\sf sing}=\{\xi\in\g^* \mid \rk \pi_t(\xi)< \rk \pi_t\} ,
\] 
which is the union of the symplectic $\g_{(t)}$-leaves in $\g^*$ having a non-maximal dimension. 
For aesthetic reasons, we write $\g^*_{\infty,\sf sing}$ instead of $\g^*_{(\infty),\sf sing}$. 

Let $\cz_t$ denote the centre of the Poisson algebra $(\gS(\g), \{\,\,,\,\}_{t})$. Then $\cz_1=\gS(\g)^{\g}$. 
For $\xi\in\g^*$, let $\textsl{d}_\xi F$ denote the differential of $F\in\gS(\g)$ at $\xi$. It is standard that for
any $H\in\gS(\g)^\g$, $\textsl{d}_\xi H\in \z(\g^\xi)$, where $\z(\g^\xi)$ is the centre of $\g^\xi$.
\\ \indent
Let $\{H_1,\dots,H_l\}$ be a set of homogeneous algebraically independent generators of $\gS(\g)^\g$.
By the {\it Kostant regularity criterion\/} for 
$\g$~\cite[Theorem~9]{ko63}, 
\beq     \label{eq:ko-re-cr}
\text{ $\left<\textsl{d}_\xi H_j \mid 1\le j\le l\right>_{\bbk}=\g^\xi$ \ if and only if \ $\xi\in\g^*_{\sf reg}$.} 
\eeq
(Recall that $\g^\xi=\z(\g^\xi)$ if and only if $\xi\in\g^*_{\sf reg}$~\cite[Thm.\,3.3]{p03}.)
Set $\textsl{d}_\xi \cz_t=\left<\textsl{d}_\xi F\mid F\in\cz_t\right>_{\bbk}$. 
Then $\textsl{d}_\xi \cz_t \subset \ker \pi_t(\xi)$ for each $t$. 
The regularity criterion obviously  holds for any $t\ne 0,\infty$. 
That is, 
\beq          \label{span-dif}
\text{if }\ t\ne0,\infty, \ \text{ then }\  \xi\not\in\g^*_{(t),\sf sing} \  
\Leftrightarrow \  
\textsl{d}_\xi \cz_t =\ker \pi_t(\xi) \Leftrightarrow \ \dim \ker \pi_t(\xi)=\rk\g . 
\eeq
A certain analogue of this statement holds for $t=0$, i.e., for $\g_{(0)}$ and $\textsl{d}_x\cz_0$, but 
only for involutions $\sigma$ such that $\gS(\g)^\g$ has a {\sf g.g.s.} for $(\g,\g_0)$, see~\cite{contr}.  

The centres $\cz_t$ ($t\in \bbk$) generate a Poisson-commutative subalgebra with respect to any 
bracket $\{\,\,,\,\}_t$, $t\in\BP$, cf. Corollary~\ref{cl-ort}. Write $\gZ=\mathsf{alg}\langle\cz_t\rangle_{t\in\bbk}$ for this subalgebra.
Note that $\textsl{d}_\xi\gZ$ is the linear span of $\textsl{d}_\xi \cz_t$ with $t\ne\infty$. 
There is a method for estimating the dimension of such subspaces, see Appendix~\ref{sect:app}.  

\begin{lm}             \label{lm-sum-reg}
Suppose that $\xi\in\g^*$ satisfy the  properties:
\begin{itemize}
\item[\bf (1)] $\dim\ker \pi_t(\xi)=\rk\g$ for all $t\ne \infty$; 
\item[\bf (2)] the rank of the skew-symmetric form $\pi_0(\xi)\vert_{\ker \pi_\infty(\xi)}$ equals
$\dim\ker \pi_\infty(\xi)-\rk\g$.
\end{itemize}
Then $\dim \textsl{d}_\xi\gZ=\rk\g+\frac{1}{2}\rk \pi_\infty(\xi)$
and $\dim \bigl(\textsl{d}_\xi\gZ\cap\ker \pi_\infty(\xi)\bigr)=\rk\g$.
\end{lm}
\begin{proof}
By definition, $\textsl{d}_\xi\gZ \subset \sum_{t\ne \infty} \ker \pi_t(\xi)$. Then
Eq.~\eqref{span-dif} and hypothesis {\bf (1)} on $\xi$ imply that 
$\textsl{d}_\xi\gZ\supset \sum_{t\ne 0,\infty} \ker \pi_t(\xi)$. Observe that we have 
a $2$-dimensional vector space of skew-symmetric bilinear forms 
$a{\cdot}\pi_t(\xi)$ on $\g\simeq T^*_\xi \g^*$, where $a\in\bbk$, $t\in \BP$. Moreover, 
$\rk \pi_t(\xi)=\dim\g-\rk\g$ for each $t\ne \infty$. 
By Lemma~\ref{open}, we have $\sum_{t\ne 0,\infty} \ker \pi_t(\xi)=\sum_{t\ne \infty} \ker \pi_t(\xi)$. 
Now the desired equalities follow from Theorem~\ref{sum-dim}.  
\end{proof} 

It is not clear yet whether such elements $\xi\in\g^*$ actually exist! However, we will immediately see that there are plenty of them.
\begin{prop}          \label{lm-trdeg}
The hypotheses of Lemma~\ref{lm-sum-reg} hold for generic $\xi\in\g^*$ and therefore
\[
\trdeg \gZ=\frac{1}{2}\rk \pi_\infty+\rk\g=\frac{1}{2}(\dim\g -\crk\{\,\,,\,\}_\infty)+\rk\g .
\]
\end{prop}
\begin{proof}
The first task is to prove that a generic point $\xi=\xi_0+\xi_1\in\g^*$ satisfies condition {\bf (1)} in Lemma~\ref{lm-sum-reg}. 

One can safely assume that $\xi$ is regular for $\{\,\,,\,\}_0$ and $\{\,\,,\,\}_\infty$. 
Next, we are lucky that  
$\xi_0+\xi_1\in \g^*_{\sf sing}=\g^*_{(1),\sf sing}$ if and only if 
$\xi_0+s^{-1}\xi_1\in\g^*_{(s^2),\sf sing}$. Therefore,
\beq      \label{eq-cdt}
   \bigcup_{t\ne 0,\infty} \g^*_{(t),\sf sing} = 
   \{\xi_0+t \xi_1 \mid \xi_0+\xi_1\in\g^*_{\sf sing}, \, t\ne 0,\infty\} .
\eeq
Since $\codim \g^*_{(t),\sf sing}=3$ for each $t\in\bbk^\times$, the closure of 
$\bigcup_{t\ne 0,\infty} \g^*_{(t),\sf sing}$ is a proper subset of $\g^*$.  Hence the condition 
$\dim\ker \pi_t(\xi)=\rk\g$ ($t\ne \infty$) holds for $\xi$ in a dense open subset. 

The next step is to check condition {\bf (2)}, i.e., compute the rank of the restriction of $\pi_0(\xi)$ to 
$\ker \pi_{\infty}(\xi)$. Write $\xi=\xi_0+\xi_1$, where $\xi_i\in\g_i^*$.    
We can safely assume that $\xi_0$ is regular in $\g$ and hence also in $\g_0$. 
\\ 
\textbullet \ \ For the inner involutions, one has $\rk\g=\rk\g_0$. Here $\ker \pi_{\infty}(\xi)=\g_0$ and the rank 
in question is $\dim\g_0-\rk\g$, as required in Lemma~\ref{lm-sum-reg}{\bf (2)}. 
\\  
\textbullet \ \
Suppose that $\sigma$ is outer. 
Then $\ker \pi_{\infty}(\xi)=\g_0 \oplus \g_1^{\xi_0}$ with  $\dim\g_1^{\xi_0}=\rk\g-\rk\g_0$. 
The rank of the form $\pi_0(\xi_0)$ on this kernel is equal to 
\[
   \dim\ker \pi_\infty(\xi_0)-\rk\g_0-\dim\g_1^{\xi_0}=\dim\ker \pi_\infty(\xi_0)-\rk\g. 
\] 
For a generic $\xi$, where $\xi_1$ is generic as well, the value in question cannot
be smaller than $\dim\ker \pi_\infty(\xi_0)-\rk\g$. On the other hand, it cannot   
be larger by Lemma~\ref{restr}.  That is, we have obtained the required value again!  

Now, it follows from Lemma~\ref{lm-sum-reg} that  
\[
\trdeg\gZ=\max_{\xi\in\g^*}\dim\textsl{d}_\xi\gZ
=\frac{1}{2}(\dim\g -\crk\{\,\,,\,\}_\infty)+\rk\g .  \qedhere
\]
\end{proof}

Combining Lemma~\ref{lm:rank-ot-t} and Proposition~\ref{lm-trdeg}, we obtain
\beq      \label{eq:formula}
   \trdeg \gZ=\frac{1}{2}(\dim\g_1+\rk\g+\rk\g_0) .
\eeq

\begin{lm}[{\cite[Prop.~1.1]{m-y}}]     \label{lm:kosik}
If $\ca\subset \gS(\g)^{\g_0}$ \ and \ $\{\ca,\ca\}_{\sf LP}=0$, then
\[      
      \trdeg\ca\le \bb(\g)-\bb(\g_0)+\ind\g_0.
\]
\end{lm}
Note that in our situation, $\bb(\g)-\bb(\g_0)+\ind\g_0=\frac{1}{2}(\dim\g_1+\rk\g+\rk\g_0)$.

\begin{lm}    \label{lm:vnutri-lezhit}
We have $\gZ\subset \gS(\g)^{\g_0}$.
\end{lm}
\begin{proof}
For all Poisson brackets $\{\,\,,\,\}_t$ with $t\ne\infty$,
the commutators $[x_0,y]$ are the same as in $\g$. Hence $\cz_t\subset \gS(\g)^{\g_0}$ for each
$t\ne \infty$.
\end{proof}

{\sl A posteriori}, this lemma is true for $\g=\tri$ as well, cf. Example~\ref{ex:sl2}.
Combining previous formulae, together with computations for $\tri$, we obtain the next general assertion.

\begin{thm}    \label{prop:main2}
For any $\g$ and any $\sigma$, the algebra 
$\gZ=\mathsf{alg}\langle\cz_t\rangle_{t\in\BP_{\sf reg}}$ is a Poisson-commutative subalgebra of\/ 
$\gS(\g)^{\g_0}$ of the maximal possible transcendence degree, which is given by 
Eq.~\eqref{eq:formula}. 
\end{thm}

In Section~\ref{sect:3}, we provide an explicit set of generators of $\gZ$, if $\gS(\g)^\g$ has a good 
generating system for $(\g,\g_0)$. From this, we deduce that $\gZ$ is a polynomial algebra. Although 
$\gZ$ has the maximal transcendence degree among the Poisson-commutative subalgebras of 
$\gS(\g)^{\g_0}$, it is not always maximal.  In Section~\ref{sect:4}, we construct the extended algebra 
$\tilde\gZ$ such that $\gZ\subset\tilde\gZ\subset \gS(\g)^{\g_0}$ and show that $\tilde\gZ$ is maximal 
and still polynomial.

\section{The algebra $\gZ$ is polynomial whenever $\sigma$ is good}
\label{sect:3}
\noindent
Let $\{H_1,\dots,H_l\}$, $l=\rk\g$,  be a set of homogeneous algebraically independent generators of 
$\gS(\g)^\g$.
Set $d_i=\deg H_i$. Then $\sum_{i=1}^l d_i=\bb(\g)$. Associated with the vector space decomposition
$\g=\g_0\oplus\g_1$, one has the bi-homogeneous decomposition of each $H_j$:
\[
    H_j=\sum_{i=0}^{d_j} (H_j)_{(i,d_j-i)} ,
\] 
where $(H_j)_{(i,d_j-i)}\in \gS^i(\g_0)\otimes \gS^{d_j-i}(\g_1)\subset \gS^{d_j}(\g)$. Let $H_j^\bullet$ be the nonzero bi-homogeneous component of $H_j$ with
maximal $\g_1$-degree. Then $\deg_{\g_1}\! H_j=\deg_{\g_1}\! H_j^\bullet$ and we set 
$d_j^\bullet=\deg_{\g_1}\! H_j^\bullet$.

\begin{df}   
Let us say that $H_1,\dots,H_l$ is a {\it good generating system} in $\gS(\g)^\g$
({\it {\sf g.g.s.}\/} for short) for $(\g,\g_0)$ or for $\sigma$, if $H_1^\bullet,\dots,H_l^\bullet$ are 
algebraically independent. 
\end{df}
\noindent
If the pair  $(\g,\g_0)$ is indecomposable, which we always tacitly assume, then there is no {\sf g.g.s.} 
for four involutions related to $\g$ of type $\eus{E}_{n}$~\cite[Remark\,4.3]{coadj} and a {\sf g.g.s.}
exists in all other cases, see \cite{contr}. The importance of {\sf g.g.s.} is clearly seen in the following
fundamental result.

\begin{thm}[{\cite[Theorem\,3.8]{contr}}]    \label{thm:kot14}  
Let $H_1,\dots,H_l$ be an arbitrary set of homogeneous algebraically independent generators of\/ $\gS(\g)^\g$. Then
\begin{itemize}
\item[\sf (i)] \ $\sum_{j=1}^l \deg_{\g_1}\!\! H_j\ge \dim\g_1$;
\item[\sf (ii)] \  $H_1,\dots,H_l$ is a {\sf g.g.s.} if and only if\/ $\sum_{j=1}^l \deg_{\g_1}\!\! H_j=\dim\g_1$; 
\item[\sf (iii)] \ if $H_1,\dots,H_l$ is a {\sf g.g.s.}, then
$\gS(\g_{(0)})^{\g_{(0)}}=\bbk[H_1^\bullet,\dots,H_l^\bullet]$ is a polynomial algebra.
\end{itemize}
\end{thm}

Recall that $\g_{(0)}=\g_0\ltimes\g_1^{\sf ab}$ is a $\BZ_2$-contraction of $\g$ and  $\ind\g_{(0)}=\ind\g$.
We continue to assume that $\g\ne\tri$, hence $\BP_{\sf reg}=\bbk$ and 
$\gZ=\mathsf{alg}\langle\cz_t\rangle_{t\in \bbk}$.
\begin{thm}    \label{thm:main3-1}
Suppose that $\{H_i\}$ is a {\sf g.g.s.} 
for $\sigma$. Then the algebra $\gZ$ is generated by
\beq   \label{eq:bihom}
   \{(H_j)_{(i,d_j-i)} \mid  j=1,\dots,l \ \& \ i=0,1,\dots,d_j\}, 
\eeq
i.e., by all bi-homogeneous components of $H_1,\dots,H_l$.
\end{thm} 
\begin{proof}
To begin with,  $\cz(\{\,\,,\,\}_1)=\cz(\gS(\g))=\bbk[H_1,\dots,H_l]$. By the definition of 
$\{\,\,,\,\}_t$, we have $\cz(\{\,\,,\,\}_t)=\varphi_{s}^{-1} (\cz(\gS(\g)))$ for $t\ne 0,\infty$, where $s^2=t$ and 
\[
    \varphi_s(H_j)=(H_j)_{(d_j,0)}+s (H_j)_{(d_j-1,1)}+ s^2 (H_j)_{(d_j-2,2)}+\dots
\]
Using the Vandermonde determinant, we deduce from this that all $(H_j)_{(i,d_j-i)}$ belong to $\gZ$ and the algebra generated by them contains $\cz_t$ with $t\in \bbk\setminus\{0\}$.
Moreover, the specific bi-homogeneous components $H_1^\bullet,\dots,H_l^\bullet$ generate $\cz_0$, since $H_1,\dots,H_l$ is a {\sf g.g.s.}
Therefore, the polynomials \eqref{eq:bihom} generate the whole of $\gZ$.
\end{proof}

However, not every $i\in \{0,1,\dots, d_j\}$ provides a nonzero bi-homogeneous 
component of $H_j$. Let us make this precise. 
Since the case of inner involutions is technically easier, we consider it first.

\begin{thm}   \label{thm:main3-2}
Suppose that $\sigma\in {\sf Aut}(\g)$ is inner, and let $H_1,\dots,H_l$ be a {\sf g.g.s.} in $\gS(\g)^\g$ with 
$d_j^\bullet=\deg_{\g_1}\!\! H_j$. Then
\begin{itemize}
\item[\sf (i)]  all $d^\bullet_j$, $j=1,\dots,l$, are even; 
\item[\sf (ii)]  $(H_j)_{(i,d_j-i)}\ne 0$ if and only if $d_j{-}i$ is even and $0\le d_j{-}i\le d_j^\bullet$;
\item[\sf (iii)]  the polynomials $\{(H_j)_{(i,d_j-i)} \mid j=1,\dots,l; \ \& \ d_j{-}i=0,2,\dots, d^\bullet_j \}$
freely generate $\gZ$.
\end{itemize}
\end{thm}
\begin{proof}
{\sf (1)} Since $\sigma$ is inner, $\sigma(H_j)=H_j$ for all $j$. On the other hand, 
$\sigma\vert_{\g_0}={\sf id}$, $\sigma\vert_{\g_1}=-{\sf id}$, and hence 
$\sigma((H_j)_{(i,d_j-i)})=(-1)^{d_j-i}(H_j)_{(i,d_j-i)}$. This yields {\sf (i)} and one implication in {\sf (ii)}.

{\sf (2)} In view of part {\sf (1)}, the number of non-zero bi-homogeneous components of $H_j$ is at
most $(d_j^\bullet/2) +1$. Hence the total number of nonzero bi-homogeneous components of all $H_j$ 
is at most $\sum_{j=1}^l (d_j^\bullet/2) +1=(\dim\g_1/2)+\rk\g$.

As $\sigma$ is inner, one also has $\rk\!\g=\rk\!\g_0$. Therefore, $\trdeg\gZ=(\dim\g_1/2)+\rk\g$, see Eq.~\eqref{eq:formula}. Because the bi-homogeneous components of all  $H_j$ generate $\gZ$ (Theorem~\ref{thm:main3-1}),
we see that all $(H_j)_{(i,d_j-i)}$ with $d_j-i=0,2,\dots, d_j^\bullet$ are nonzero and algebraically independent. Thus, they freely generate $\gZ$.
\end{proof}

With extra technical details, Theorem~\ref{thm:main3-2} extends to the outer involutions as well.
Let $\sigma$ be an arbitrary involution of $\g$. It is easily seen that a set of homogeneous generators
of $\gS(\g)^\g$ can be chosen so that each $H_j$ is an eigenvector of $\sigma$, i.e.,
$\sigma(H_j)=\esi_j H_j=\pm H_j$. Moreover, the set of pairs $\{(d_j,\esi_j)\mid j=1,\dots,l\}$ does not 
depend on the set of generators, cf.~\cite[Lemma\,6.1]{springer}. However, we need a set of free generators 
that both is a {\sf g.g.s}{.} and consists of $\sigma$-eigenvectors.

\begin{lm}     \label{kosiks-lemma}
If there is a {\sf g.g.s.} for $(\g,\g_0)$, then there is also a {\sf g.g.s.} that consists of eigenvectors of  $\sigma$. 
\end{lm}
\begin{proof}
Let $H_1,\dots,H_l$ be a {\sf g.g.s.}, hence 
$\sum_{j=1}^l \deg_{\g_1}H_j=\dim\g_1$ in view of Theorem~\ref{thm:kot14}. 

Let $\gA_+$ be the ideal in $\gS(\g)^\g$ generated by all homogeneous invariants of positive degree.
Then $\boldsymbol{\gA}:=\gA_+/\gA_+^2$ is a finite-dimensional $\bbk$-vector space. If $H\in \gA_+$, 
then $\bar H:=H+\gA_+^2\in \boldsymbol{\gA}$.
As is well-known, $F_1,\dots,F_m$ is a generating system for $\gS(\g)^\g$ if and only if 
the $\bbk$-linear span of $\bar F_1,\dots, \bar F_m$ is the whole of $\boldsymbol{\gA}$.
In our situation, $\dim_\bbk \boldsymbol{\gA}=l$ and $\boldsymbol{\gA}=\langle
\bar H_1,\dots, \bar H_l\rangle$.

If $H_i$ is not a $\sigma$-eigenvector, i.e., $\sigma(H_i)\ne\pm H_i$, then we consider the generating set
\[
   H_1,\dots,H_{i-1}, \frac{H_i+\sigma(H_i)}{2}, \frac{H_i- \sigma(H_i)}{2}, H_{i+1},\dots, H_l 
\] 
for  $\gS(\g)^\g$ that includes $l+1$ polynomials. Since 
$\bar H_,\dots,\bar H_{i-1}, \bar H_{i+1},\dots, \bar H_l$ are linearly independent in $\boldsymbol{\gA}$,
we obtain a better generating set by replacing $H_i$ with one of the functions
$H_i^{(+)}=\frac{H_i+\sigma(H_i)}{2}$ or $H_i^{(-)}=\frac{H_i-\sigma(H_i)}{2}$. Let us demonstrate that there is actually only one suitable replacement for $H_i$, and this yields again a {\sf g.g.s.} 
Recall that $d_j^\bullet=\deg_{\g_1}\! H_j^\bullet=\deg_{\g_1}\! H_j$. 
 
(a)  Suppose that $d_j^\bullet$ is even. Then $\sigma(H_i^\bullet)=H_i^\bullet$ and $H_i^\bullet$ cancel out in 
$H_i^{(-)}$. Therefore, $\deg_{\g_1}H_i^{(-)}< \deg_{\g_1}H_i$ and the sum of $\g_1$-degrees for $H_1,\dots,H_{i-1},H_i^{(-)}, H_{i+1},\dots, H_l$ is less than $\dim\g_1$. By Theorem~\ref{thm:kot14},
this means that the choice of 
$H_i^{(-)}$ in place of $H_i$ does not provide a generating system, and the only right choice is to take
$H_i^{(+)}$. Moreover, $H_i^\bullet=(H_i^{(+)})^\bullet$ and hence
$H_1,\dots,H_{i-1}, H_i^{(+)},  H_{i+1},\dots, H_l$ is a {\sf g.g.s.}
 
(b)  If $d_j^\bullet$ is odd, then we end up with the {\sf g.g.s.}
$H_1,\dots,H_{i-1}, H_i^{(-)},  H_{i+1},\dots, H_l$.

The procedure reduces the number of generators that are not $\sigma$-eigenvectors, and we eventually
obtain a {\sf g.g.s.} that consists of $\sigma$-eigenvectors. 
\end{proof}

Without loss of generality, we can assume that $H_1,\dots,H_l$ is a {\sf g.g.s.} and 
$\sigma(H_j)=\pm H_j$.

\begin{lm}    \label{lm:outer-inv}
For any involution $\sigma\in {\sf Aut}(\g)$, we have

{\bf (1)} \ $\sigma(H_j)=H_j$ if and only if $d^\bullet_j$ is even;

{\bf (2)} \ $\rk\g_0=\# \{j \mid \sigma(H_j)=H_j\}$.
\end{lm}
\begin{proof}
{\sf (1)} \ The proof is similar to that of Theorem~\ref{thm:main3-2}(i).
\\ 
{\sf (2)} \ This follows from results of T.\,Springer on regular elements of finite reflection 
groups~\cite[Corollary\,6.5]{springer}. To this end, one has to consider the Weyl group corresponding
to a Cartan subalgebra
$\te=\te_0\oplus \te_1\subset\g_0\oplus\g_1$ such that $\te_0$ is a Cartan in $\g_0$.
\end{proof}

Now, we can state and prove the main result of this section.
\begin{thm}\label{free-main}
Let $\sigma$ be an involution of $\g$ such that $\gS(\g)^\g$ has a {\sf g.g.s.}
Then $\gZ$ is a polynomial algebra that is freely generated by the bi-homogeneous components
of all $\{H_j\}$. More precisely, if $\sigma(H_j)=H_j$, then $d^\bullet_j$ is even and
the nonzero bi-homogeneous components of $H_j$ are $(H_j)_{(i,d_j-i)}$ with 
$d_j-i=0,2,\dots, d^\bullet_j$; 
if $\sigma(H_j)=-H_j$, then $d^\bullet_j$ is odd and
the nonzero bi-homogeneous components of $H_j$ are $(H_j)_{(i,d_j-i)}$ with 
$d_j-i=1,3,\dots, d^\bullet_j$.
\end{thm}
\begin{proof}
By Lemma~\ref{lm:outer-inv}, we may order the basic invariants $\{H_j\}$ such that
\[
   d_j^\bullet \ \text{ is } \begin{cases} \text{ even } & i\le k:=\rk\g_0 ; \\  \text{ odd } & i\ge k+1 .
   \end{cases} 
\]
Clearly, if $d_j^\bullet$ is even, then $\esi_j=1$ and  $H_j$ has at most $(d_j^\bullet/2)+1$ nonzero 
bi-homogeneous components, while if $d_j^\bullet$ is odd, then $\esi_j=-1$ and $H_j$ has at most 
$(d_j^\bullet+1)/2$ nonzero bi-homogeneous components.
Hence the total number of all nonzero bi-homogeneous components is at most
\[
  \sum_{j=1}^k \left(\frac{d_j^\bullet}{2}+1\right)  +  \sum_{j=k+1}^l \frac{d_j^\bullet+1}{2}=
  \sum_{j=1}^l \frac{d_j^\bullet}{2}+k +\frac{l-k}{2}=\frac{\dim\g_1+\rk\g+\rk\g_0}{2}=\trdeg \gZ .
\]
Therefore, all admissible bi-homogeneous components must be  nonzero and algebraically independent.
\end{proof}

\begin{rmk}    \label{rem:Z-for-bad-sigma}
If there is no g.g.s{.} for $(\g,\g_0)$, then $\sum_j \deg_{\g_1}H_j >\dim\g_1$ for any set of basic invariants.
Hence the number of the bi-homogeneous components of $\{H_j\}$ is bigger than $\trdeg\gZ$ and
these generators of $\gZ$ are algebraically dependent. Moreover, the algebra 
$\cz_0=\cz(\gS(\g_0\ltimes\g_1^{\sf ab}))$, which is contained in $\gZ$, is not 
polynomial~\cite[Section\,6]{Y-imrn}, 
and also $H_1^\bullet,\dots,H_l^\bullet$ are algebraically dependent, cf. Theorem~\ref{thm:kot14}.
Thus, we cannot say anything good about $\gZ$ in the four ``bad'' cases.
\end{rmk}

\begin{rmk}      \label{rem:sigma-inner}
Recall from the introduction the map $r_0: \gS(\g)^\g \to  \gS(\g_0)^{\g_0}$.
If $\sigma$ is inner, then $\g_0$ contains a Cartan subalgebra of $\g$ and $r_0$ is injective. 
Hence $(H_j)_{(d_j,0)}=r_0(H_j)\ne 0$ for all $j$,
which also  follows from Theorem~\ref{thm:main3-2}. Clearly, 
$r_0(\gS(\g)^\g)\subset \gZ$ for any $\sigma$. More precisely,  $r_0(\gS(\g)^\g)$ is freely generated by
the $r_0(H_j)=(H_j)_{(d_j,0)}$ such that $\sigma(H_j)=H_j$ (i.e., $d_j^\bullet$ is even).
However, for the inner (and some outer) involutions, $r_0(\gS(\g)^\g)$ is a proper
subalgebra of $\gS(\g_0)^{\g_0}$. And this is the reason, why $\gZ$ appears to be not always a maximal 
commutative subalgebra of $\gS(\g)^{\g_0}$.
\end{rmk}

\section{The extended algebra $\tilde\gZ$ is polynomial and maximal Poisson-commutative}
\label{sect:4}

\noindent
In this section, we assume that $\g\ne\gt{sl}_2$, 
$(\g,\g_0)$ is indecomposable, and there is a {\sf g.g.s.} for $(\g,\g_0)$. We write $\z(\q)$ for the centre
of a Lie algebra $\q$.
An open subset of $\g^*$ is said to be {\it big}, if its complement does not contain divisors.
 
There is an extraordinary powerful tool for proving maximality of certain subalgebras.

\begin{thm}[{\cite[Theorem~1.1]{trio}}]       \label{ppy-max}  
Let $F_1,\ldots,F_r\in\gS(\g)$ be homogeneous algebraically independent polynomials 
such that their differentials $\{\textsl{d}F_i\}$ are linearly independent 
on a big open subset of $\g^*$. Then  $\bbk[F_1,\ldots,F_r]$ is an 
algebraically closed subalgebra of $\gS(\g)$, i.e., if $H\in \gS(\g)$ is 
algebraic over the field\/ $\bbk(F_1,\ldots,F_r)$, then
$H\in \bbk[F_1,\ldots,F_r]$. 
\end{thm}

In order to apply this theorem to $\gZ$ and $\tilde\gZ$, we need some properties of divisors in $\g^*$.

\begin{lm}\label{d-easy}
Let $D\subset \g^*$ be an irreducible divisor. Then there is a non-empty open subset 
$U\subset D$ such that, for each $\xi\in U$, we have 
\begin{itemize}
\item[{\sf (i)}] $\xi\not\in \g^*_{(t),\sf sing}$, if $t\ne \infty$;
\item[{\sf (ii)}] if $\xi=\xi_0+\xi_1$ with $\xi_i\in\g_i^*$, then $\xi_0\in(\g_0^*)_{\sf reg}$.
\end{itemize}
\end{lm}
\begin{proof}
{\sf (i)} \ The Lie algebra $\g_{(0)}=\g_0\ltimes\g_1^{\sf ab}$ has the 
{\sl codim}--$2$ property, see \cite[Theorem\,3.3]{coadj}. Hence $\codim \g^*_{(0),\sf sing}\ge 2$.
Recall that 
$\dim\g^*_{\sf sing}=\dim\g-3$.  Therefore, the union of the singular subsets 
$\g^*_{(t),\sf sing}$, $t\in \bbk^\times$, 
is a subset of codimension $2$, as follows from Eq.~\eqref{eq-cdt}. Hence
there is a non-empty open subset of $D$ such that 
$\rk \pi_t(\xi)=\rk \pi_t$ for each $\xi\in D$ and  $t\ne \infty$. 

{\sf (ii)} \ Since  $\g_0$ is reductive, we also have  
$\dim (\g_0^*)_{\sf sing}\le \dim\g_0-3$. 
\end{proof}

\begin{lm}      \label{lm-dif-z} 
Suppose that the differentials $\{\textsl{d}(H_j)_{(i,d_j{-}i)}\}$ are linearly dependent 
on an irreducible divisor $D\subset \g^*$. Then $D \subset \g^*_{\infty,\sf sing}$. 
\end{lm}
\begin{proof}
Combining Lemmas~\ref{lm-sum-reg} and \ref{d-easy},  we see that if the differentials of the 
$(H_j)_{(i,d_j{-}i)}$'s  are linearly dependent at a generic point $\xi\in D$, then 
\\ \indent
-- either $\rk \pi_\infty(\xi) < \rk \pi_{\infty}$, 
\\ \indent
-- or $\rk \pi_\infty(\xi) = \rk \pi_{\infty}$, but 
the restriction of $\pi_0(\xi)$ to 
$\ker \pi_{\infty}(\xi)$ does not have the prescribed (maximal possible) rank. 

In the first case, we have $\xi\in \g^*_{\infty,\sf sing}$ by the very definition. Let us show that
the second possibility does not realise. Write $\xi=\xi_0+\xi_1$. By Lemma~\ref{d-easy}{\sf (ii)}, 
we may assume that $\xi_0\in \g_{0,\sf  reg}$. Since $\rk \pi_\infty(\xi) = \rk \pi_{\infty}$, we also have 
$\xi_0\in\g_{\sf reg}$.  As in the proof of Proposition~\ref{lm-trdeg}, 
the rank of $\pi_0(\xi_0)$ on $\ker \pi_{\infty}(\xi)$ equals $\dim\ker \pi_{\infty}(\xi_0)-\rk\g$. 
And again  the same holds for the restriction  of $\pi_0(\xi)$.  
\end{proof} 

We also need the following simple but useful observation on $ \g^*_{\infty,\sf sing}$.
\begin{lm}      \label{lm:D_0-and-g_reg}
The subvariety $\g^*_{\infty,\sf sing}$ is of the form $X_0\times \g_1^*$, where $X_0\subset \g_0^*$ is
a conical subvariety. Moreover, $X_0\cap \g^*_{\sf reg}=\varnothing$.
\end{lm}
\begin{proof}
Let $\xi=\xi_0+\xi_1\in\g^*$. Since $\g^\xi_{(\infty)}=\g_0\oplus \g_1^{\xi_0}$, the value 
$\rk\pi_\infty(\xi)$ depends only on $\xi_0=\xi\vert_{\g_0}$. 
Therefore, $\g^*_{\infty,\sf sing}=X_0\times \g_1^*$, where $X_0=\g^*_{\infty,\sf sing}\cap\g_0^*$.

It follows from the proof of Lemma~\ref{lm:rank-ot-t} that  $\min_{\xi_0\in\g_0}\dim\g_1^{\xi_0}=
\rk\g-\rk\g_0$, and $\xi\in \g^*_{\infty,\sf sing}$ if and only if 
$\dim\g_1^{\xi_0} >\rk\g-\rk\g_0$. But, if $\xi_0\in\g^*_{\sf reg}$, then
$\dim\g_0^{\xi_0}=\rk\g_0$ and $\dim\g_1^{\xi_0}=\rk\g-\rk\g_0$.
\end{proof}
A particularly nice situation occurs if 
      $ r_0\!: \gS(\g)^{\g}\to \gS(\g_0)^{\g_0}$
is onto. This condition is rather restrictive. If $\sigma$ is inner, then $\bb(\g)=\bb(\g_0)+(\dim\g_1)/2$. 
And since $\sum_{j=1}^l d_j=\bb(\g)$, the nonzero polynomials $\{(H_j)_{(d_j,0)}\}_{j=1}^l$  cannot form a 
generating system in $\gS(\g_0)^{\g_0}$. Hence $r_0$ cannot be onto for the inner $\sigma$.
Another observation is that $\g_0$ has to be simple. This leads to the following list of suitable 
symmetric pairs:
\beq          \label{cs-list}
  (\h \oplus \h,\h), \ (\gt{sl}_n,\gt{so}_n),   \ 
  (\gt{sl}_{2n},\gt{sp}_{2n}), \ (\gt{so}_{2n},\gt{so}_{2n{-}1}), \ 
  (\eus{E}_6,\gt{sp}_8), \ (\eus{E}_6,\eus{F}_4)
\eeq
Among them the map $r_0$ is onto for $(\gt h \oplus \gt h,\gt h)$,
$(\gt{sl}_{2n+ 1},\gt{so}_{2n+ 1})$, 
$(\gt{sl}_{2n},\gt{sp}_{2n})$, $(\gt{so}_{2n},\gt{so}_{2n{-}1})$, and $(\eus{E}_6,\eus{F}_4)$. 
But, the pair $(\eus{E}_6,\eus{F}_4)$ is not needed, because it does not have a {\sf g.g.s.}  

\begin{thm}   \label{thm:maxim1}
{\sf (1)} \ If the restriction homomorphism $r_0:  \gS(\g)^{\g}\to \gS(\g_0)^{\g_0}$ is onto, then 
$\g^*_{\infty,\sf sing}$ does not contain divisors and $\gZ$ is a maximal Poisson-commutative subalgebra 
of\/ $\gS(\g)^{\g_0}$.
\\ \indent
{\sf (2)} \ Conversely, if $\gZ$ is maximal Poisson-commutative, then $r_0$ is onto.
\end{thm}
\begin{proof}
(1) \ The list of suitable symmetric pairs is quite short. For each item in the list, $\g_0$ contains a 
nilpotent element that is regular in $\g$. This implies that {\bf every} fibre of the quotient morphism 
$\g_0\to \g_0\md G_0$ contains a regular element of $\g$ and hence 
 $(\g_0)^*_{\sf reg}\subset \g^*_{\sf reg}$. Thus, 
$\dim(\g^*_{\sf sing}\cap\g_0)\le \dim\g_0-3$. Since $\rk \pi_{\infty}(\xi)=\rk \pi_{\infty}$ for each 
$\xi\in\g_0^*\cap \g^*_{\sf reg}$ (Lemma~\ref{lm:D_0-and-g_reg}), the subset 
$\g^*_{\infty,\sf sing}$ does not contain divisors. Therefore, the differentials $\textsl{d}(H_j)_{(i,d_j{-}i)}$ 
are linearly independent on a big open subset, in view of Lemma~\ref{lm-dif-z}. 
Then, by Theorem~\ref{ppy-max}, $\gZ$ is an algebraically closed subalgebra of $\gS(\g)$. Since it is 
a Poisson-commutative subalgebra of $\gS(\g)^{\g_0}$ of the maximal possible transcendence degree, 
it is also maximal. 

(2) \ If $r_0$ is not onto, then the algebra generated by $\gS(\g_0)^{\g_0}$ and $\gZ$ is
Poisson-commutative, is contained in $\gS(\g)^{\g_0}$, and properly contains $\gZ$.
\end{proof} 

\begin{rmk}     \label{equiv} 
{\sf (1)} \  Consider the following four conditions: 
\begin{itemize}
\item[{\sf (a)}] the restriction homomorphism $r_0\!: \gS(\g)^{\g}\to \gS(\g_0)^{\g_0}$ is onto; 
\item[{\sf (b)}] $\g _0$ contains a regular nilpotent element of $\g$; 
\item[{\sf (c)}] $\g^*_{\infty,\sf sing}$ does not contain divisors; 
\item[{\sf (d)}] $\gZ$ is a maximal Poisson-commutative subalgebra of $\gS(\g)^{\g_0}$.
\end{itemize}
In the  proof of Theorem~\ref{thm:maxim1}(1), we have seen that (a) $\Rightarrow$ (b) $\Rightarrow$ (c) 
$\Rightarrow$ (d), whereas part (2) of Theorem~\ref{thm:maxim1} states that (d) $\Rightarrow$ (a).
Thus, all these conditions are equivalent. One can also give a direct proof for (b) $\Rightarrow$ 
(a) that does not invoke $\g_{(\infty)}$ and $\gZ$. 
However, the implication (a) $\Rightarrow$ 
(b)  is obtained case-by-case as yet.
\\ \indent
{\sf (2)}  \  There is a {\sf g.g.s.} for $(\g,\g_0)$ if and only if the restriction homomorphism 
$r_1\!: \gS(\g)^{\g}\to \bbk[\g_1^*]^{\g_0}$ is onto~\cite{coadj,contr}. Therefore, $\gZ$ is a 
polynomial maximal 
Poisson-commutative subalgebra of $\gS(\g)^{\g_0}$ whenever both $r_0$ and $r_1$ are onto. 
\end{rmk}

Our ultimate goal is to prove that, in general, $\tilde\gZ=\mathsf{alg}\langle \gZ, \gS(\g_0)^{\g_0}\rangle$ 
is a polynomial maximal Poisson-commutative subalgebra of $\gS(\g)^{\g_0}$. Unfortunately, the proof 
requires many technical preparations, if 
$\g^*_{\infty,\sf sing}$ contains divisors (i.e., $r_0$ is not onto).

\begin{lm}    \label{lm-div} 
Suppose that $\dim \g^*_{\infty,\sf sing} = n{-}1$, and 
let $D\subset  \g^*_{\infty,\sf sing}$ be an irreducible component of dimension $n-1$.
Then   
\begin{itemize}
\item[\sf (i)] \ $D=D_0\times\g_1^*$, where $D_0$ is a $G_0$-stable conical divisor in $\g_0^*$, and 
$D_0$ does not contain regular elements of $\g$;
\item[\sf (ii)] \  generic elements of $D_0$ are semisimple, regular in $\g_0\simeq\g_0^*$, and 
subregular in $\g$;
\item[\sf (iii)] \ 
$\rk \pi_{\infty}(\xi)=\rk \pi_{\infty}{-}2$ \ for generic point $\xi\in D$.
\end{itemize} 
\end{lm}
\begin{proof}
{\sf (i)} \ This follows from Lemma~\ref{lm:D_0-and-g_reg}. 

{\sf (ii)} \ If $\sigma$ is \un{inner}, then $\g_0$ contains a Cartan subalgebra $\gt t$ of $\g$ and 
$\te\cap D_0$ is a $W_0$-stable divisor in $\te$, where $W_0$ is the Weyl group of $(\g_0,\te)$.
It is easily seen that any such divisor contains a subregular element of $\g$. 

The case of an  \un{outer} $\sigma$ is more involved. 
We use an argument, which is also valid for the inner case. 
If $\te_0\subset\g_0$ is a Cartan subalgebra of $\g_0$, then a generic element 
$\nu\in D_0\cap\gt t_0$  is either regular or subregular in $\g_0$.  Consider these two possibilities in turn.

(a) \ Suppose first that $\nu$ is regular in $\g_0$. Then $\g_{0}^\nu=\te_0$ and therefore 
$\g^\nu$ is a sum of a toral subalgebra and several copies, say $k$, of $\tri$. 
Let $\es_i$ be the $i$-th copy of $\tri$. 
Every such $\es_i$ is determined by a root $\beta_i$ of $\g$. That is, 
\[
     \es_i=\g_{-\beta_i}\oplus (\es_i)^\sigma \oplus \g_{\beta_i} .
\]
Moreover, the one-dimensional subspace $(\es_i)^{\sigma}$ is generated by the 
coroot $\beta_i^\vee$. It is also clear that $(\es_i)^{\sigma}\subset \te_0$ and $\beta_i^\vee$ is
orthogonal to $\nu$. Assume that $k\ge 2$. Then $\nu$ is orthogonal to at least two different coroots.
Since the number of relevant pairs $\{\beta_i,\beta_j\}$ is finite, we obtain that $D_0\cap\gt t_0$ lies in 
a finite union of subspaces of $\te_0$ of codimension~${\ge}2$. A contradiction! Hence 
$k\le 1$. If $k=0$, then $\nu$ is regular in $\g$, which is impossible, see {\sf (i)}. Thus,
$k=1$ and $\nu$ is subregular in $\g$. 
\\ \indent
(b) \ Suppose now that $D_0$ 
does not contain regular semisimple elements of $\g_0$. Our goal is to prove that this case does not occur.  \\ \indent
Here  $\te_0\cap D_0$ is a union of reflection hyperplanes of $W_0$. Let $\z_0$ be one of these 
hyperplanes and $\nu\in\z_0$ generic.  Then 
$\g_{0}^\nu=\es \oplus \z_0$, where $\es\simeq\tri$ and $\z_0$ is the centre of $\g_{0}^\nu$.
Here  $[\z_0,\g^\nu]=0$, since $\nu\in\z_0$ is generic. 
Write $\g^\nu=\h\oplus \z(\g^\nu)$, where $\h=[\g^\nu,\g^\nu]$ is semisimple. Then the symmetric pair 
$(\g^\nu,\g_{0}^\nu)$ decomposes as
\[
    (\g^\nu,\g_{0}^\nu)=(\h,\es)\oplus (\z(\g^\nu),\z_0).
\] 
The only possibilities for the symmetric pair $(\h,\es)$ are:
\beq     \label{pairs}
   (\gt{sl}_2 \oplus \gt{sl}_2,\gt{sl}_2),  \  
   (\gt{sl}_3,\gt{so}_3\simeq\tri), \ 
   (\gt{sl}_2,\gt{sl}_2). 
\eeq
For $\es=[\g_{0}^\nu,\g_{0}^\nu]$, the intersection 
$D_0\cap (\bbk \nu\oplus \es)$ is a conical divisor of $\bbk \nu\oplus \es$ that contains $\nu$.  
If $\eta\in\es$ is non-zero semisimple, 
then $\nu+\eta\in (\g_0)_{\sf reg}$ is semisimple. Hence
$\nu+\eta\not\in D_0$.  Therefore, 
$D_0\cap (\bbk \nu\oplus \es)$ has to contain a sum $\nu+e$, where $e\in\es $ is regular nilpotent. 
For all pairs in~\eqref{pairs},   
$e$ is also regular in $\h$. Hence $e$ is a regular element of $\g^\nu$. 
Thereby  $\nu+e$ is a regular element of $\g$. However, this contradicts part {\sf (i)}.

Therefore, case (b) does not materialise and, according to (a),
 $D_0$ contains a semisimple element $\nu$ that is regular in $\g_0$ and subregular in $\g$. Since
 $D_0\cap\g_{\sf reg}=\varnothing$, subregular semisimple elements of $\g$ are dense in $D_0$.

{\sf (iii)} \ Since $\nu$ is regular in $\g_0$  and subregular in $\g$, we have $\dim\g_{0}^\nu=\rk\g_0$ and
$\dim\g_{1}^\nu = \rk\g + 2-\rk\g_0$. The latter precisely means that
$\rk \pi_{\infty}(\nu)=\rk \pi_{\infty}{-}2$ for $\nu$ in a non-empty open subset of $D_0$.
This completes the proof.   
\end{proof}

\begin{ex}         \label{ex-sl1}
Let $(\g,\g_0)=(\gt{sl}_{2n},\gt{so}_{2n})$. Then $D_0\subset \g_0$ is the zero set of the Pfaffian. If 
$\g_0$ consists of skew-symmetric matrices with respect to the antidiagonal, then \\
\centerline{$x={\rm diag}(a_1,\ldots,a_{n{-}1},0,0,-a_{n{-}1},\ldots,-a_1)\in D_0$} 
\\ is subregular whenever all $a_i$ are nonzero and $a_i\ne \pm a_j$ for $i\ne j$.
\end{ex}

Recall that $\{H_i\}$ is a {\sf g.g.s.} for $\sigma$ such that $\sigma(H_i)=\esi_i H_i=\pm H_i$ for each $i$. 
As before, $d_i=\deg H_i$ and $l=\rk\g$. 
Until the end of this section, we assume that $d_1\le \cdots \le d_l$.
If $\g$ is simple, then there is a unique basic invariant of degree $d_l$, i.e., $d_{l-1}< d_l$.

\begin{lm}    \label{subreg-1}
If $\g$ is simple and $x\in\g$ is subregular, then the differentials $\{\textsl{d}_x H_i \mid i<l\}$ are 
linearly independent. Moreover, $\sigma(H_l)=H_l$ unless $(\g,\g_0)= (\gt{sl}_{2k+ 1},\gt{so}_{2k+ 1})$, 
where $l=2k$ and $d_l=2k+1$.
\end{lm}
\begin{proof}
Let $e\in\g$ be a subregular nilpotent element. Then $\textsl{d}_e H_l=0$~\cite[Corollary\,2]{va} and 
$\{\textsl{d}_e H_i\mid i<l\}$ are linearly independent~\cite[Chapter\,8.2]{815}.  If $x$ is subregular and
non-nilpotent,  then
the theory of {\it associated cones\/} developed in 
\cite[\S\,3]{bokr} shows that $Ge\subset \overline{\bbk^\times\! (G x)}$. 
This implies that $\textsl{d}_x H_i$ with $i<l$ are linearly independent, too. 

The equality $\sigma(H_l)=H_l$ is obvious for the inner involutions. If $\sigma$ is outer, 
then going through the list of outer involutions, one checks that $\sigma(H_l)=-H_l$ if and only if 
$\g=\gt{sl}_{2k+1}$ and $l=2k$. Here necessary $\g_0=\gt{so}_{2k+1}$.
\end{proof}

We need below some formulae for the differential and partial derivatives of a homogeneous polynomial
$F\in\gS(\g)=\bbk[\g^*]$. If $x\in \g^*$ and $d=\deg F$, then $\partial^{d-1}_x F$ is a linear form on
$\g^*$, i.e., an element of $\g$. In fact, one has
\beq    \label{diff-versus-partial}
    (d-1)! \, \textsl{d}_x F= \partial^{d-1}_x F.
\eeq 
By linearity, it suffices to check this for a monomial of degree $d$.
Furthermore, for the operator $\partial_{x+sx'}^k: \gS^m(\g)\to \gS^{m-k}(\g)$ with $x,x'\in\g^*$ and 
$s\in\bbk$, there is the following expansion:
\beq   \label{eq:partial}
  \partial^k_{x{+}sx'} =\partial_x^k + \binom{k}{1}s\partial_{x'}\partial_x^{k{-}1}+\dots
   + \binom{k}{i} s^i \partial_{x'}^i \partial^{k{-}i}_x + \dots + s^k\partial^k_{x'} \ .
\eeq

\begin{lm}                 \label{subreg}
Suppose that the restriction homomorphism $r_0$ is not onto 
(equivalently, $\g^*_{\infty,\sf sing}$ contains divisors). Then
\begin{itemize}
\item[\sf (i)] \ there is $x\in\g_0^*\simeq\g_0$ such that $x$ is semisimple, 
regular in $\g_0$, and subregular in $\g$ (i.e., $\dim\g^x=\rk\g+2$). Moreover, 
for a generic $x'\in\g_1^*\simeq\g_1$,  we have 
$y:=x+x'\in \g_{\sf reg}$;
\item[\sf (ii)] \ $\lim_{t\to\infty} \left<\textsl{d}_y F \mid F\in\cz_t\right>_{\bbk}=\lim_{s\to 0} \varphi_{s}(\g^{x+sx'}) (=:\BV)$;
\item[\sf (iii)] \ $\dim(\BV /\BV\cap \g_0)=\rk\g-\rk\g_0+1$.  
\end{itemize}
\end{lm}
\begin{proof}
{\sf (i)} \ The existence of such an $x$ follows from Lemma~\ref{lm-div}. Then
$\g^x_0=\te_0$ and if $x'$ is a generic element of $\g^x_1$, then $y$ is regular in $\g$.
Hence $x+x'\in\gt g_{\sf reg}$ for almost all $x'\in\gt g_1$.

{\sf (ii)} \ By the definition of $\{\,\,,\,\}_{t}$, we have $\cz_t=\varphi_s^{-1}(\gS(\g)^{\g})$ if
$t\ne 0,\infty$ and $s^2=t$. Let $\varphi_s^*\!: \g^*\to \g^*$ be the dual map, i.e., 
$\varphi_s^*|_{\g_0^*}=\id$, $\varphi_s^*|_{\g_1^*}=s^{-1}{\cdot}\id$. 
For any $F\in \gS(\g)$, we have $\varphi_s(\textsl{d}_y F)=\textsl{d}_{\varphi_s^*(y)} \varphi_s(F)$. 
In particular, $$\textsl{d}_y \varphi^{-1}_s(H_i)=\varphi_s^{-1}(\textsl{d}_{\varphi_s^*(y)} H_i),$$
where $\varphi_s^*(y)=x+s^{-1}x'$. 
If $s$ tends to $\infty$, then $s^{-1}$ tends to $0$. It remains to notice, that
for almost all $s$, the element $x+sx'$ is regular and then
$\g^{x+sx'}$ is the linear span of $\{\textsl{d}_{x+sx'} H_j\}_{j=1}^l$, see Eq.~\eqref{eq:ko-re-cr}.  

{\sf (iii)} \ The hypothesis that $r_0$ is not onto excludes the pairs $(\gt h\oplus\gt h,\gt h)$ and
$ (\gt{sl}_{2k+ 1},\gt{so}_{2k+ 1})$. Hence $\g$ is simple and, by Lemma~\ref{subreg-1}, 
 $\textsl{d}_x H_1,\dots, \textsl{d}_x H_{l-1}$ are linearly independent,   
$\textsl{d}_x H_l$ is a linear combination of $\textsl{d}_x H_j$ with $j<l$, and $\sigma(H_l)=H_l$. Since
$x$ is semisimple and subregular, $\g^x=\z(\g^x)\oplus \tri$ and $\dim\z(\g^x)=l-1$. Hence
$\z(\g^x)=\langle \textsl{d}_x H_i \mid i<l\rangle_{\bbk}$.  

Take $j<l$ and set $m_j=d_j-1$.  Then by Eq.~\eqref{diff-versus-partial} and by  
Eq.~\eqref{eq:partial} with $k=m_j$, we have
\begin{gather*}
   (m_j)!\,\textsl{d}_{x+sx'} H_j = \partial^{m_j}_{x+sx'} H_j=
   \sum_{i=0}^{m_j} \binom{m_j}{i} s^i \partial_{x'}^i \partial^{m_j- i}_x  H_j , \\
   (m_j)!\,\sigma(\textsl{d}_{x+sx'} H_j)
   = \partial^{m_j}_{x-sx'}\sigma(H_j)=
   \sum_{i=0}^{m_j} \binom{m_j}{i} (-s)^i \partial_{x'}^i \partial^{m_j- i}_x \sigma(H_j) .
\end{gather*}
It follows that $\partial_{x'}^i \partial_x^{m_j-i}H_j\in\g_0$ if and only if either $i$ is even and 
$\sigma(H_j)=H_j$ or $i$ is odd and $\sigma(H_j)=-H_j$. Therefore, 

\textbullet \quad if
$\sigma(H_j)=H_j$, then  
$\lim_{s\to 0} \varphi_s(\textsl{d}_{x{+}sx'} H_j) = \textsl{d}_{x} H_j\in \g_0$; while 

\textbullet \quad if $\sigma(H_j)=-H_j$, then $\textsl{d}_{x} H_j\in \g_1$ and 
\begin{multline*}     
(m_j)!\,\varphi_s(\textsl{d}_{x{+}sx'} H_j)=s (\partial^{m_j}_x H_j+m_j\partial_{x'} \partial^{m_j-1}_x H_j) +   \text{(terms of degree $\ge 2$ w.r.t.~$s$)} \\
=s \bigl( (m_j)!\, \textsl{d}_{x} H_j+ m_j\partial_{x'} \partial^{m_j-1}_x H_j\bigr) + \dots
\end{multline*} 
Thus, if $\sigma(H_j)=-H_j$, then  
\beq  \label{eq:sigma-minus}
\lim_{s\to 0} \langle \vp_s(\textsl{d}_{x{+}sx'} H_j)\rangle_\bbk=
\langle \textsl{d}_{x} H_j+\frac{1}{(m_j{-}1)!} {\cdot} \partial_{x'} \partial^{d_j{-}2}_x H_j\rangle_\bbk\,.
\eeq
Note that here $\partial_{x'} \partial^{d_j{-}2}_x H_j\ \in \g_0$. 
Write $\z(\g^x)=\z(\g^x)_0 \oplus \z(\g^x)_1$, where 
$\z(\g^x)_i = \z(\g^x)\cap \g_i$. Then 
$\z(\g^x)_1=\langle \textsl{d}_x H_j \mid \sigma(H_j)=-H_j\rangle_{\bbk}$ and 
$\z(\g^x)_0=\langle \textsl{d}_x H_j \mid \sigma(H_j)=H_j, \ j\ne l\rangle_{\bbk}$. Hence
$\dim\z(\g^x)_0=\rk\g_0-1$ and $\dim\z(\g^x)_1=\rk\g-\rk\g_0$.

Let ${\bf p}_1$ denote the projection $\g\to \g_1$ along $\g_0$. Then ${\bf p}_1(\BV)=
\BV /(\BV\cap \g_0)$ and 
our goal is to compute $\dim{\bf p}_1(\BV)$. By Eq.~\eqref{eq:sigma-minus}, we have
$\z(\g^x)_1 \subset {\bf p}_1(\BV)$. 

For our further argument, some properties of $\textsl{d}_{x} H_l\in \g_0^x$ are needed. It would be 
nice to have $\textsl{d}_{x} H_l=0$ for $x$ as in {\sf (i)}. Since this is not always the case, we need 
a trick.     

Let $\tilde{\g}=\g\oplus\ce$ be the central extension of $\g$, where $\dim\ce=1$. We extend the 
$\BZ_2$-grading to $\tilde\g$ so that  $\ce\subset\tilde{\g}_0$ and 
 $\varphi_s$ to $\tilde{\g}$ by letting $\varphi_s\vert_\ce=\id$. 
Take non-zero $z\in \ce$ and $\gamma\in\ce^*$.
Note that $\tilde{\g}^{y{+}\gamma}=\tilde{\g}^y$ 
for any $y\in\g^*$.   
Therefore $\BV\oplus\ce = \lim\limits_{s\to 0} \varphi_s( \tilde{\g}^{x+ \gamma+ sx'})$. 
Set $\bx =x+\gamma$. Then $\bx\in \tilde\g^*$ is still subregular and $z(\bx )\ne 0$. 
Clearly, there is a linear combination 
\[
  {\bf H}_l=H_l + c_{l{-}1} z^{d_l-d_{l-1}} H_{l{-}1} + \ldots + c_j z^{d_l{-}d_j} H_j + \ldots + c_0 z^{d_l}
\]
with $c_i\in\bbk$ such that 
$\partial_{\bx} ^{d_l{-}1} {\bf H}_l=\textsl{d}_{\bx} {\bf H}_l=0$. Note that $z, H_1,\dots,H_{l-1}, {\bf H}_l$
freely generate ${\cz}\gS(\tilde{\g})$.

Let ${\ca}_{\bx}$ be the Mishchenko--Fomenko subalgebra of $\gS(\tilde\g)$ associated with $\bx$.
By definition, ${\ca}_{\bx}$ is generated by 
\beq     \label{eq:gener_A-zeta}
  \{z, \ \partial_{\bx}^k H_j \ (j<l,\ 0\le k\le m_j), \ \partial_{\bx}^k {\bf H}_l \ (0\le k\le m_l-1)\}.
\eeq
As the total number of these generators is $\bb(\g)$ and $\trdeg\,{\ca}_{\bx}=\bb(\tilde\g)-1=\bb(\g)$~\cite[Lemma~2.1]{m-y}, we see that ${\ca}_{\bx}$
is freely generated by them. Note that the set in~\eqref{eq:gener_A-zeta} contains a basis for
the $l$-dimensional space $\z(\g^x)\oplus\ce=\z(\tilde\g^{\bx})$. Therefore, 
$F=\partial_{\bx}^{m_l{-}1} {\bf H}_l$ does not lie in $\gS^2\bigl(\z(\g^x)\oplus\ce\bigr)$. 
Since $\partial_{\bx}^{m_l} {\bf H}_l=0$, the polynomial $F$  is a $\tilde\g^{\bx}$-invariant in 
$\gS^2(\tilde\g^{\bx})$~\cite[Lemma~1.5]{m-y}. It is clear that $\sigma(F)=F$ and therefore 
$F\in\gS^2(\tilde\g_0)\oplus\gS^2(\g_1)$. Now $\tilde\g^{\bx}=\ce\oplus\g^x=\ce\oplus\z(\g^x) \oplus \gt{sl}_2$. There is a standard basis $\{e,h,f\}$ of this $\gt{sl}_2$ such that 
$e,f\in\g_1$ (cf. Example~\ref{ex:sl2}) and $F\in (4ef+h^2) + \gS^2(\gt z(\g^x)\oplus \ce)$.

If  $x'\in\g_1^*$ is generic enough, then $\partial_{x'}F=\eta+\xi$, 
where $\eta\in\gt z(\g^x)_1$ and $\xi$ is a non-zero element in  
$\left<e,f\right>_{\bbk}\subset\g_1$. Note that in this case 
$(m_l{-}1)!\,\textsl{d}_{{\bx}  + sx'} {\bf H}_l$ lies in $ s\partial_{x'}F+s^2\tilde{\g}$. 
Further,
\[
    (m_l-1)!  \varphi_s(\textsl{d}_{{\bx}  + sx'} {\bf H}_l)= s^2(\eta+\xi) + 
     \frac{m_l-1}{2} s^2 \partial^2_{x'} \partial_{{\bx}}^{d_l-3} {\bf H}_l 
     + \text{(terms of degree $\ge 3$ w.r.t.~$s$)}.
\]
Here $\partial^2_{x'} \partial_{{\bx}}^{d_l-3} {\bf H}_l \in\g_0$. 
Hence ${\bf p}_1(\BV)=\gt z(\g^x)_1 + \bbk(\eta+\xi)=\gt z(\g^x)_1\oplus\bbk\xi$. 
The desired equality $\dim {\bf p}_1(\BV)=\rk\g -\rk\g_0+1$ follows. 
\end{proof}

\begin{lm}            \label{subreg-rk}
Let $y=x+x'$ be as in Lemma~\ref{subreg} with $x'$ generic. 
Then the rank of the restriction of 
$\pi_0(y)$ to $\ker\pi_\infty(y)=\g_0\oplus\g_{1}^x$ is equal to $\dim(\g_0 \oplus \g_{1}^x) -\rk\g$. 
\end{lm}
\begin{proof}
Set $U=\ker \pi_{\infty}(x)=\g_0\oplus\g_{1}^x$. 
Consider the maximal torus $\gt t=\g_{0}^x+\bbk^{l-1}$, where $\bbk^{l{-}1}$ is the centre of $\g^x$. 
The intersection $\gt t_0=\g_0\cap\gt t=\g_{0}^x$ is a Cartan subalgebra of $\g_0$. Further, 
$\g_0=\te_0 \oplus \gt m$, where $\gt m$ is the $\te_0$-stable complement of $\gt t_0$ in $\g_0$. 
The torus $\gt t$ defines a finer decomposition of $U$, namely 
\[
   U=\gt m \oplus \te_0 \oplus (\g_\alpha \oplus \g_{-\alpha}) \oplus \z 
\]
where $\g_{\pm\alpha}$ are root spaces and $\z\simeq \bbk^{l{-}\rk\g_0}$. 

Choose a very particular $x'$, namely as  $x'=\xi_\alpha-\xi_{-\alpha}$ 
with non-zero root vectors $\xi_\alpha\in\g_{\alpha}$, $\xi_{-\alpha}\in\g_{-\alpha}$ under the usual identification $\g_1\simeq \g_1^*$.  Then the matrix of $(\pi_0(y)|_{U})$ with respect to a basis for $U$
adapted to the above finer decomposition
has a block form with easy to understand blocks (see Fig.~\ref{fig:block}):
\begin{itemize}
\item \ $\pi_0(y)$ is non-degenerate on $\gt m$;
\item \ $\pi_0(y)(\gt m,\gt t_0 \oplus \g_\alpha \oplus \g_{-\alpha})=0$;
\item \ $\pi_0(y)(\gt t_0,\g_\alpha \oplus \g_{-\alpha})\ne 0$.
\end{itemize}
\begin{center}
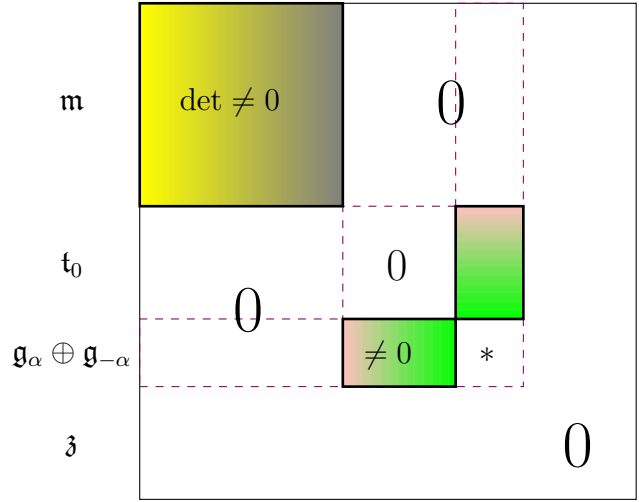
\begin{figure}[htbp]
\begin{tikzpicture}[scale= .6]
\draw (0,0)  rectangle (11,11);

\draw[dashed,redi]  (0,2.5) -- (8.5,2.5) -- (8.5,11) -- (7,11) -- (7,4) -- (0,4)--cycle;
\draw[dashed,redi]  (4.5,2.5) -- (4.5,6.5) -- (8.5,6.5);

\shade[left color=yellow,right color=gray, draw,line width=1pt] (0,6.5) -- (4.5,6.5) -- (4.5,11) -- (0,11)--cycle ;
\shade[left color=pink,right color=green, draw,line width=1pt] (4.5,2.5) -- (7,2.5) -- (7,4) -- (4.5,4)--cycle ;
\shade[top color=pink,bottom color=green, draw,line width=1pt] (8.5,6.5) -- (7,6.5) -- (7,4) -- (8.5,4)--cycle ;

\draw (2,8.8)  node {$\det\ne 0$} ;
\draw (5.5,3.2)  node {$\ne 0$} ;
\draw (-1.5,8.8)  node {$\mathfrak m$} ;
\draw (-1.5,5.2)  node {$\mathfrak t_0$} ;
\draw (-1.5,3.2)  node {$\mathfrak g_\ap\oplus \mathfrak g_{-\ap}$} ;
\draw (-1.5,1.2)  node {$\mathfrak z$} ;
\draw (5.7,5.2)  node {{\Large $0$}} ;
\draw (7.7,3.2)  node {$\ast$} ;
\draw (2.4,4.2)  node {{\Huge ${0}$}} ;
\draw (6.9,8.8)  node {{\Huge ${0}$}} ;
\draw (9.7,1.2)  node {{\Huge ${0}$}} ;
\end{tikzpicture}
\caption{The block structure of $\pi_0(y)\vert_U$}   \label{fig:block} 
\end{figure}
\end{center}

This is enough to see that the rank of $\pi_0(y)$ on $\g_0 \oplus \g_\alpha \oplus \g_{-\alpha}$ 
is at least $\dim\g_0-\rk\g_0+2$. Hence $\rk(\pi_0(y)|_{U})\ge \dim U- \rk\g$. 
This happens for one, not exactly generic $x'$, however, the generic value cannot be smaller and 
it also cannot be larger by Lemma~\ref{restr}. 
\end{proof}

The algebra $\gS(\g_0)^{\g_0}$ is contained in the Poisson centre of 
$\gS(\g)^{\g_0}$. Let $\tilde\gZ$ be the (Poisson-commutative) subalgebra of $\gS(\g)$ generated by $\gZ$ and 
$\gS(\g_0)^{\g_0}$. If $H_1,\ldots,H_l$ is a {\sf g.g.s.} for $(\g,\g_0)$ such that $\sigma(H_i)=\pm H_i$ 
for each $i$, then $\tilde\gZ$ is freely generated by $(H_j)_{(i,d_j-i)}$ with $i\ne d_j$ and a set of basic 
invariants $\tilde H_1,\ldots,\tilde H_{\rk\g_0}\in\gS(\g_0)^{\g_0}$. In other words,  a set of basic invariants of $\tilde\gZ$ is obtained from that of $\gZ$ if one replaces the generators of $r_0(\gS(\g)^\g)$ with the free generators of 
$\gS(\g_0)^{\g_0}$. [Recall that $(H_j)_{(d_j,0)}\ne 0$ if and only if $\esi_j=1$ and there are $\rk\g_0$ such indices $j$, see Remark~\ref{rem:sigma-inner}.]
 
\begin{thm}      \label{thm:maxim2}
{\sf (i)} The differentials of the algebraically independent generators of $\tilde\gZ$, chosen among 
$\{(H_j)_{(i,d_j-i)}\}$ and $\{\tilde H_j\}$, as above, are linearly independent on a big open subset of $\g^*$. 
\par 
{\sf (ii)} The algebra $\tilde\gZ$ is a maximal Poisson-commutative subalgebra of\/ $\gS(\g)^{\g_0}$.
\end{thm}
\begin{proof}
{\sf (i)} Assume that the differentials of the chosen algebraically independent generators of $\tilde\gZ$ 
are linearly dependent at each point $y$ of an irreducible divisor $D\subset\g^*$. Since 
$\gZ\subset \tilde\gZ$, the same holds for $\textsl{d} (H_j)_{(i,d_j-i)}$. Then 
$D\subset\g^*_{\infty,\sf sing}$ by Lemma~\ref{lm-dif-z} and $D=D_0\times \g_1^*$ by Lemma~\ref{lm-div}{\sf (i)}.  Let $y=x+x'$ be a generic element of $D$. 

Recall that $\textsl{d}_y\gZ$ stands for the linear span of $\textsl{d}_y F$ with $F\in\gZ$. 
We have $$\textsl{d}_y\gZ=\sum_{t\ne \infty} \textsl{d}_y \cz_t.$$ 
According to Lemma~\ref{d-easy}, $y\in \g^*_{(t),\sf reg}$ for each $t\ne \infty$. 
Hence $\textsl{d}_y \cz_t=\ker\pi_t(y)$ whenever $t\ne \infty$. 
By Lemma~\ref{lm-div}{\sf (iii)}, $\rk\pi_{\infty}(y)=\rk\pi_{\infty}-2$. 
Combining Lemmas~\ref{lm-div}{\sf (ii)},  \ref{subreg}{\sf (i)}, and  \ref{subreg-rk}, 
we see that the   rank of the restriction of 
$\pi_0(y)$ to $\ker\pi_\infty(y)$ is equal to $\dim\ker\pi_\infty(y) -\rk\g$. 
Now Theorem~\ref{sum-dim} applies and asserts that 
\beq    \label{eq:half}
    \dim (\textsl{d}_y\gZ/ \textsl{d}_y\gZ \cap \ker \pi_\infty(y)) = \frac{1}{2} \rk \pi_{\infty}(y)=\frac{1}{2}\rk \pi_{\infty} -1. 
\eeq 
By construction, $\BV \subset \textsl{d}_y \gZ \cap \ker \pi_{\infty}(y)$. Recall that 
$\ker\pi_\infty(y)=\g_0\oplus\g_1^x$. In view of \eqref{eq:half} and 
 Lemma~\ref{subreg}{\sf (iii)}, we have 
$$\dim(\textsl{d}_y \gZ /\textsl{d}_y\gZ \cap \g_0) \ge
\frac{1}{2}\rk \pi_{\infty} -1 + (\rk\g-\rk\g_0+1) = 
  \frac{1}{2}\rk \pi_{\infty} + \rk\g-\rk\g_0.$$  
The differentials  $\{\textsl{d}_x\tilde H_j\mid j=1,\dots,\rk\g_0\}$ are linearly independent and lie in 
$\g_0$. 
Hence 
$$
\dim\textsl{d}_y \gZ\ge  \frac{1}{2}\rk \pi_{\infty} + \rk\g-\rk\g_0 + \rk\g_0 = \trdeg \gZ.
$$
Now we see that the differentials of all the generators of $\tilde\gZ$ are linearly independent 
at $y$. A contradiction! 

Part {\sf (ii)}  follows from {\sf(i)} and Theorem~\ref{ppy-max}.  
\end{proof}

\noindent
{\bf Remark.} In the jargon of completely integrable systems, which is used e.g. in~\cite{mf, bols},
Eq.~\eqref{eq:half} means that the restriction of $\gZ$ to the symplectic leaf of $\{\,\,,\,\}_{\infty}$ at $y$ is a ``complete family in involution".

\section{Fancy identities for Poisson tensors}    \label{fancy}
\label{sect:fancy}

In this section, the existence of a {\sf g.g.s.} is of no importance, any indecomposable symmetric pair
$(\g,\g_0)$ is admitted. 

Let $\omega$ be the standard $n$-form on $\g^*$, where $n=\dim\g$, and let $\pi$ be the Poisson 
tensor (bivector) of the Lie--Poisson bracket on $\g^*$, see Section~\ref{subs:P-tensor}.
Having a basis $\{e_1,\ldots,e_n\}$ for $\g$, one can write 
$$
\pi=\sum_{i<j} [e_i,e_j]\otimes \partial_i \wedge \partial_j, \ \text{ where } \ \partial_i=\partial_{e_i}.
$$
For simplicity, we identify $\cW^1$ with $\gS(\g)\otimes \g^*$ and $\partial_i$ with $e_i^*$, 
where $e_i^*$ are the elements of the dual basis $\{e_1^*,\ldots,e_n^*\}\subset \g^*$.  
We also identify $de_i$ with $e_i$ and therefore $\Omega^1$ with $\gS(\g)\otimes \g$.

For any $k>0$, set
$$
\bigwedge\!^k\pi=\,\underbrace{\pi\wedge\pi\wedge\ldots\wedge \pi}_{k \ \scriptstyle{\mathrm{factors}}}
$$
and regard it as
an element of $\gS^k(\g)\otimes\bigwedge^{2k}\g^*$.
Then $\bigwedge^{(n-l)/2}\pi\ne 0$ and all higher
exterior powers of $\pi$ are zero.  There is a formula describing $\bigwedge^{(n-l)/2}\pi$ 
in terms of the Poisson centre of $\gS(\g)$. 
Applying  the map $\vp_s^{-1}$, one 
obtains a similar formula for $\vp_s^{-1}(\pi)$, which is the Poisson tensor of $\{\,\,,\,\}_s$, 
in terms of the Poisson centre 
of $(\gS(\g),\{\,\,,\,\}_s)$.  The main idea of \cite{contr} was to consider the minimal $s$-components
of both sides. Here we consider the maximal $s$-components and obtain interesting new identities.

By definition, $\textsl{d}F \in \Omega^1$  for each $F\in\gS(\g)$. 
Take $H_1,\ldots,H_l\in\gS(\g)^{\g}$.  Then
$$
\textsl{d}H_1\wedge\ldots\wedge \textsl{d}H_l\in  \gS(\g)\otimes\bigwedge\!^l \g.
$$
At the same time, $\bigwedge^{(n-l)/2}\pi\in\gS(\g) \otimes\bigwedge^{n-l}\g^*$.
The volume form  $\omega$
defines a non-degenerate pairing between $\bigwedge^l\g$ and
$\bigwedge^{n-l}\g$. If $u\in \bigwedge^l\g$ and $v\in \bigwedge^{n-l}\g$,
then $u\wedge v=c\,\omega$ with $c\in\bbk$. We write this as
$\dfrac{u\wedge v}{\omega}=c$ and let $\dfrac{u}{\omega}$ be the element of
$(\bigwedge^{n-l}\g)^*$ such that $\dfrac{u}{\omega}(v)=\dfrac{u\wedge v}{\omega}$.
For any ${\bf u}\in \gS(\g){\otimes} \bigwedge^l\g$,
we let $\dfrac{\bf u}{\omega}$ be the corresponding element of
\[
\gS(\g)\otimes\left({\bigwedge}\!^{n-l}\g\right)^*\cong \gS(\g)\otimes{\bigwedge}\!^{n-l}\g^*.
\]
There is a Poisson interpretation of the Kostant regularity criterion \cite[Theorem\,9]{ko63}, see also
Eq.~\eqref{eq:ko-re-cr},
the so-called {\it Kostant identity} (see~\cite{contr}):
\[
    \frac{\textsl{d} H_1\wedge\dots \wedge \textsl{d} H_l}{\omega}={\bigwedge}\!^{(n-l)/2} \pi .
\]
The identity holds if the basic invariants are normalised correctly. It still holds if we apply $\vp_s^{-1}$ 
to both sides. 

Suppose that $\sigma$ is outer and  $\sigma(H_j)=-H_j$. Then 
\[  
  \textsl{d} (H_{j})_{(d_j {-}1,1)}\in \underbrace{\gS^{d_j{-}1}(\g_0)\otimes {\bigwedge}\!^{1}\g_1}_{I} 
  \oplus \underbrace{\g_1\gS^{d_j{-}2}(\g_0)\otimes{\bigwedge}\!^{1}\g_0}_{II}.
\]
Let $\textsl{d}H_{j}^{[1]}$ stand for the component of the first type. This is a $1$-form on $\g^*$.
Suppose that $\sigma(H_i)=H_i$ for $i\le k$ and $\sigma(H_i)=-H_i$ for $i > k$. 
Then $k=\rk\g_0$ here, cf. Lemma~\ref{lm:outer-inv}. 

Let $\pi_{\g_0}$ denote the Poisson tensor of $\g_0$. Since  $\g_0$ is reductive,
${\bigwedge}\!^{(\dim\g_0-\rk\g_0)/2} \pi_{\g_0}$ is non-zero on the big open subset $(\g^*_0)_{\sf reg}$.

\begin{prop}        \label{lm-dif-K} 
If $\sigma$ is an inner involution, then 
\beq   \label{K-inner}
 \frac{\textsl{d} (H_1)_{(d_1,0)}\wedge\dots \wedge \textsl{d}( H_l)_{(d_l,0)}}{\omega} =
  {\bigwedge}\!^{(\dim\g_1)/2}\pi_{\infty} \otimes {\bigwedge}\!^{(\dim\gt  g_0-l)/2} \pi_{\g_0}.
\eeq
If $\sigma$ is an outer involution, then 
\beq      \label{K-outer}
\begin{array}{l}
 \dfrac{\textsl{d} (H_1)_{(d_1,0)}  \wedge \dots \wedge  \textsl{d} (H_k)_{(d_k,0)} \otimes 
  \textsl{d} H_{k+ 1}^{[1]}\wedge  \dots\wedge \textsl{d}H_l^{[1]} }  {\omega} = \\
  \qquad \qquad   \qquad \qquad    \qquad \qquad    \qquad \qquad  =
  {\bigwedge}\!^{(\dim\g_1-l+k)/2}\pi_{\infty} \otimes {\bigwedge}\!^{(\dim\gt  g_0-k)/2} \pi_{\g_0}.
  \end{array}
\eeq
\end{prop}
\begin{proof}
The product $\textsl{d} H_1\wedge\dots \wedge \textsl{d} H_l$ is an $l$-form on $\g^*$ with polynomial 
coefficients. Among these coefficients, we are interested in those that have the maximal possible degree in $\g_0$. 
It is not difficult to see that the degree in question is equal to $\bb(\g){-}l=(n-l)/2$ and  that the 
corresponding $l$-form is either $\textsl{d} (H_1)_{(d_1,0)}\wedge\dots \wedge \textsl{d}( H_l)_{(d_l,0)}$
in the inner case or 
\[
  \textsl{d} (H_1)_{(d_1,0)}  \wedge \dots \wedge  \textsl{d} (H_k)_{(d_k,0)} \otimes 
  \textsl{d} H_{k+ 1}^{[1]}\wedge  \dots\wedge \textsl{d}H_l^{[1]}
\]
 in the outer case.  For the first one, we have 
\[
 \frac{\textsl{d} (H_1)_{(d_1,0)}\wedge\dots \wedge \textsl{d}( H_l)_{(d_l,0)}}{\omega} \in 
  \gS^{(n{-}l)/2}(\g_0)\otimes{\bigwedge}\!^{\dim\g_0{-}l}\g_0^*\otimes{\bigwedge}\!^{\dim\g_1}\g_1^*.
\]
In case of an outer involution $\sigma$, the $(n{-}l)$-vector belongs to 
\[
\gS^{(n{-}l)/2}(\g_0) \otimes
  {\bigwedge}\!^{\dim\g_0{-}k}\g_0^* \otimes{\bigwedge}\!^{\dim\g_1-l+k}\g_1^*.
\]
The right hand side of the Kostant identity is a polyvector with polynomial coefficients of degree 
$\bb(\g){-}l$.  If $\xi\otimes (x{\wedge}y)$ is a summand of $\pi$ and $\xi\in\g_0$, then 
either $x,y\in\g_1^*$ or $x,y\in\g_0^*$. This justifies the right hand sides of \eqref{K-inner} and \eqref{K-outer}. 
\end{proof}

If $\sigma$ is inner, then $\{(H_i)_{(d_i,0)}\}$ are algebraically independent. Hence also the right hand 
side of~\eqref{K-inner} is nonzero. In particular, $ {\bigwedge}\!^{\dim\g_1/2}\pi_\infty\ne 0$ in complete accordance with Lemma~\ref{lm:rank-ot-t}. 
 If $\sigma$ is outer, then  $ {\bigwedge}\!^{(\dim\g_1-l+k)/2}\pi_{\infty}\ne 0$ by Lemma~\ref{lm:rank-ot-t}. 
It is also clear that ${\bigwedge}\!^{(\dim\gt  g_0-k)/2} \pi_{\g_0}\ne 0$. Therefore the left hand side of
\eqref{K-outer} is nonzero, too. 

Suppose that $\sigma$ is inner. Then
${\bigwedge}\!^{(\dim\g_1)/2}\pi_{\infty}=F{\cdot}x_1{\wedge}\ldots{\wedge} x_{\dim\g_1}$, where $F\in\gS^{\dim\g_1}(\g_0)$ and $\{x_j\}$ is a basis for  $\g_1^*$. 
The zero set of $F$ is exactly $\g^*_{\infty, \sf sing}$.  Under the identifications   
$\g_0\simeq \g_0^*$, we have that $F(\xi_0)=\det (\ad(\xi_0)|_{\g_1})$ for $\xi_0\in\g_0$.

Let $\{\tilde H_1,\ldots,\tilde H_l\}$ be a set of suitably normalised basic $\g_0$-invariants in $\gS(\g_0)$.
Then they satisfy the Kostant identity  with $\bigwedge\!^{(\dim\g_0- l)/2}\pi_{\g_0}$ 
on the right hand side. In other words, if $\omega_0$ is the volume form on $\g_0^*$, then
\[
  \frac{\textsl{d} \tilde H_1 \wedge \dots \wedge  \textsl{d}  \tilde H_l}{\omega_0} = 
  \bigwedge\!^{(\dim\g_0-l)/2}\pi_{\g_0}. 
\]
Plugging this identity into \eqref{K-inner}, we obtain the following statement. 

\begin{cl}       \label{cl-inner}
Keep the assumption that $\sigma$ is inner and regard $(H_j)_{(d_j,0)}$ as an element of $\gS(\g_0)$. 
Then 
$$
\textsl{d}(H_1)_{(d_1,0)}{\wedge}\ldots{\wedge}\textsl{d}(H_l)_{(d_l,0)} = F{\cdot} \textsl{d}\tilde H_1{\wedge}\ldots{\wedge}\textsl{d}\tilde H_l,
$$ 
where $F$ is the same as above. Hence the differentials $\{\textsl{d}(H_i)_{(d_i,0)}\}$ are linearly dependent 
exactly on the subset $\g^*_{\infty,\sf sing} \cup (\g_0^*)_{\sf sing}$. 
\qed
\end{cl}

\begin{prop}      \label{centre-inf} 
Let $\sigma$ be an outer involution. Then  $(H_j)_{(d_j-1,1)}$, where $k<j\le l$,  together with a basis 
$\{\xi_1,\ldots,\xi_{\dim\g_0}\}$ of $\g_0$ freely generate $\cz_\infty$. 
Further, there is $Q\in\gS(\g_0)$ such that 
\[
Q{\cdot} \frac{\xi_1{\wedge}\ldots {\wedge} \xi_{\dim\g_0} \wedge \textsl{d}H_{k+ 1}^{[1]}{\wedge}\ldots{\wedge}
 \textsl{d} H_l ^{[1]} } {\omega} = \bigwedge\!^{(\dim\g_1-l+k)/2} \pi_\infty .
\] 
If $Q$ is regarded as a function on $\g^*$, then  its zero locus is the 
maximal divisor of $\g^*$ contained in $\g^*_{\infty,{\sf sing}}$.  \
 \end{prop}
\begin{proof}
Set $P_0=\bigwedge_{i=1}^{\dim\g_0} \xi_i$, $P_1=\bigwedge_{j=k+ 1}^{l} \textsl{d} H_{j}^{[1]}$, 
and $P=P_0 \wedge P_1$.  By the construction of $H_j^{[1]}$, we have also 
$P=P_0\wedge (\bigwedge_{j=k+ 1}^l \textsl{d} (H_j)_{(d_j{-}1,1)})$. 

Take $x\in\g_0^*$. If $\sigma(H_j)=-H_j$, then $\textsl{d}_x H_j= 
\textsl{d} H_j ^{[1]}(x) =\textsl{d}_x (H_j)_{(d_j-1,1)} \in\g_1$. 
If $y=x+x'$ with $x\in\g_0^*$, $x'\in\g_1^*$, then $P(y)=P_0\wedge P_1(x)$. 
We wish to show that $P(y)\ne 0$ on a big open subset of $\g^*$. 
This is equivalent to the {\bf claim} that $P_1(x)\ne 0$ on a big open subset of 
$\g_0^*$. 

Assume that $P_1$ is zero on an irreducible divisor $X\subset \gt g_0^*$. 
By Lemma~\ref{d-easy}{\sf (ii)}, $x\in(\gt g_0^*)_{\sf reg}$ for a generic $x\in X$. 
If $x\in\g_0^*$ is regular in $\g$, then the elements $\textsl{d}_x H_i$ with $1\le i\le l$ are linearly independent, see Eq.~\eqref{eq:ko-re-cr}, and  $P_1(x)\ne 0$. 
Thus, $\dim\gt g^x\ge l+2$ for all $x\in X$ and $X{\times}\gt g_1 \subset \g^*_{\infty,\sf sing}$. 
This settles the claim for the cases, where $r_0$ is surjective and $\g^*_{\infty,\sf sing}$ does not contain divisors. 

Suppose that $\dim \g^*_{\infty,\sf sing}=n-1$. Let $x\in X$ be generic. By Lemma~\ref{lm-div},
$\dim\gt g^x=l+2$. Lemma~\ref{subreg-1} states that 
the elements $\textsl{d}_x H_j$ with $\sigma(H_j)=-H_j$ are linearly independent.  
Thereby $P_1(x)\ne 0$. 
The claim is settled. 

By Theorem~\ref{ppy-max}, the subalgebra of $\gS(\g)$ generated by $(H_j)_{(d_j{-}1,1)}$ with $k<j\le l$
and $\xi_i$ with $1\le i \le \dim\g_0$  is algebraically closed. 
Since it  lies inside $\cz_\infty$ and has the same transcendence degree, $\dim\g_0+(l-k)$, 
it coincides with $\cz_\infty$. 

Since $P$ is non-zero on a big open subset, we have 
$$
Q{\cdot} \frac{\xi_1{\wedge}\ldots {\wedge} \xi_{\dim\g_0} \wedge \textsl{d}H_{k+ 1}^{[1]}{\wedge}\ldots{\wedge}
 \textsl{d} H_l^{[1]} } {\omega} = {\bigwedge}\!^{(\dim\g_1-l+k)/2} \pi_\infty 
$$ 
for some $Q\in\gS(\g)$, see e.g.  \cite[Section~2]{contr}. Since all the coefficients in the right hand 
side are elements of $\gS(\g_0)$, we have $Q\in\gS(\g_0)$ as well. 
\end{proof}

\begin{rmk}  \label{rem:Z-infty}
If $\sigma$ is inner, then $\trdeg\cz_\infty=\dim\g_0$ and it is easily seen that 
$\cz_\infty=\gS(\g_0)$ as subalgebra of $\gS(\g_{(\infty)})$. In particular, $\cz_\infty$ is always a polynomial algebra.
\end{rmk}

Combining Proposition~\ref{centre-inf} with Eq.~\eqref{K-outer} and the Kostant identity for 
$\g_0$, we obtain the following assertion.  

\begin{cl}         \label{cl-outer}
Let $\tilde H_1,\ldots,\tilde H_k$ be properly normalised basic
$\g_0$-invariants in $\gS(\g_0)$. Then 
$$
\textsl{d}(H_1)_{(d_1,0)}\wedge\ldots \wedge\textsl{d}(H_k)_{(d_k,0)} = Q{\cdot}\textsl{d}\tilde H_1\wedge\ldots\wedge\textsl{d}\tilde H_k 
$$ 
in $\gS(\g_0)\otimes {\bigwedge}\!^{k}\g_0$ with the same $Q$ as in Proposition~\ref{centre-inf}.
The differentials $\textsl{d}(H_1)_{(d_1,0)},\ldots,\textsl{d}(H_k)_{(d_k,0)}$ are linearly dependent exactly 
on the union of $(\g_0^*)_{\sf sing}$ with the zero set of $Q$. 
\qed
\end{cl}

Note that $Q$ is the Pfaffian in the setting of Example~\ref{ex-sl1}.

\section{Further developments and possible applications}  
\label{sect:further}
 
\noindent 
We believe that this paper is the beginning of a long exciting journey. Several applications of our 
construction are already available and are presented below.  Goals further ahead are stated as conjectures.   
 
\subsection{Quantum perspectives} 
\label{sec-q}
Let $\eus U(\g)$ be the enveloping algebra of $\g$.
Given a Poisson-commutative subalgebra $\mathcal C\subset \gS(\g)$, it is
natural  to ask whether there exists a commutative subalgebra
$\widehat{\mathcal C}\subset \eus U(\g)$ such that ${\rm gr}(\widehat{\mathcal C})={\mathcal C}$. \
This question was posed by Vinberg for the Mishchenko--Fomenko subalgebras~\cite{v:sc}, and it is 
known nowadays as {\it Vinberg's problem}. For the semisimple $\g$, the first conceptual solution 
was obtained in \cite{r:si}. The r\^ole of the symmetrisation map 
$\varpi : \gS(\g)\to {\eus U}(\g)$ in that quantisation for the classical $\g$ is explained in \cite{m-y}. 
 
\begin{conj}
Suppose that there is a {\sf g.g.s.} for $\sigma$. Let $\widehat{\gZ}$ be the subalgebra of\/ ${\eus U}(\g)$ generated by $\varpi((H_j)_{(i,d_j{-}i)})$ with $1\le i\le l$, $0\le i\le d_i$. Then 
$\widehat{\gZ}$ is commutative and\/ ${\rm gr}(\widehat{\gZ}) = \gZ$.  
\end{conj}
 
For the symmetric pairs 
$(\gt{gl}_{n+ m},\gt{gl}_n{\oplus} \gt{gl}_m)$, 
$(\gt{sp}_{2(n+ m)},\gt{sp}_{2n}{\oplus} \gt{sp}_{2m})$, and
$(\gt{so}_{n+ m},\gt{so}_n{\oplus} \gt{so}_m)$,
there might be a connection between $\widehat{\gZ}$ and commutative subalgebras of Yangians or 
twisted Yangians.

The Yangian $Y(\gt{gl}_m)$ is a deformation of the enveloping algebra $\eus U(\gt{gl}_m[z])$ of the 
current algebra $\gt{gl}_m[z]$ given by explicit generators and relations. Then
${\eus U}(\gt{gl}_m)$ is a subalgebra 
of $Y(\gt{gl}_m)$. The facts on Yangians, which are used below, can be found in \cite{book}, see in particular Chapter~8 therein. The most relevant for us is the {\it centraliser construction} of 
Olshanski~\cite{o} and Molev--Olshanski~\cite{mo}.  
For any $n$, there is an almost surjective map
$$
          \Psi_n\!: Y(\gt{gl}_m)\to {\eus U}(\gt{gl}_{n+ m})^{\gt{gl}_n}, 
$$
where the words ``almost surjective" mean that ${\eus U}(\gt{gl}_{n+ m})^{\gt{gl}_n}$ is generated by the image of $Y(\gt{gl}_m)$ and ${\eus U}(\gt{gl}_n)^{\gt{gl}_n}$. 
It is known that, for a fixed $m$, $\bigcap\limits_{n\ge 1} \ker\Psi_n=0$. 

\begin{qtn}
Is there a commutative subalgebra ${\mathcal B}\subset Y(\gt{gl}_m)$ such that 
${\rm gr}(\Psi_n({\mathcal B}))$ together with ${\mathcal Z}\gS(\g_0)$ generate $\tilde\gZ\subset \gS(\gt{gl}_m{\oplus} \gt{gl}_n)$?
\end{qtn}

Let $Y(\gt{sp}_{2m})\subset Y(\gt{gl}_{2m})$ be the twisted Yangian in the sense of G.\,Olshanski. 
Here ${\eus U}(\gt{sp}_{2m})\subset Y(\gt{sp}_{2m})$ and there is again an almost surjective map 
\[  
     \Psi_n\!: Y(\gt{sp}_{2m})\to {\eus U}(\gt{sp}_{2n+ 2m})^{\gt{sp}_{2n}}. 
\]
Then one can pose an analogous question. A similar situation occurs for 
$Y(\gt{so}_{m})\subset Y(\gt{gl}_{m})$ and $\eus U(\gt{so}_{n+m})$ with $n$ even.  
 
Any natural quantisation of $\gZ$ has to provide a commutative subalgebra 
$\widehat{\gZ} \subset {\eus U}(\g)^{\g_0}$. By adding ${\eus U}(\g_0)^{\g_0}$ 
one obtains the related quantisation $\widehat{\tilde\gZ}$ of $\tilde\gZ$.
Let $V$ be a finite-dimensional simple $\g$-module. Then $\widehat{\tilde\gZ}$ acts on 
the subspace $V^{\gt n_0}\subset V$ of the highest  weight vectors of $\g_0$.  

\begin{conj} \label{conj-spec}
Let  $\widehat{\tilde\gZ}\subset\eus U(\g)$ be the subalgebra generated by 
$\varpi((H_j)_{(i,d_j{-}i)})$ with $1\le i\le l$, $0\le i\le d_i$ and by $\eus U(\g_0)^{\g_0}$. 
Then  $\widehat{\tilde\gZ}$ acts on $V^{\gt n_0}$  diagonalisably  and with a simple spectrum. 
\end{conj}

If Conjecture~\ref{conj-spec} is true, then the action of $\widehat{\tilde\gZ}$
produces a solution of the branching problem  $\g \downarrow \g_0$.  
There are two renowned examples, where both conjectures are true. 
 
\begin{ex}[{The Gelfand--Tsetlin construction \cite{gt-1,gt-2}}]      \label{G-Ts-ex}
Let $(\g,\g_0)$ be one of the symmetric pairs $(\gt{sl}_{n+ 1},\gt{gl}_n)$, $(\gt{so}_{n+ 1},\gt{so}_n)$.
Then each $H_i$ has at most two nonzero bi-homogeneous components. 
To be more precise, the Pfaffian in the case of $\g=\gt{so}_{2l}$ has one nonzero component, and 
all the other generators have exactly two. It follows that $\tilde\gZ$ is generated by 
$\gS(\g_0)^{\g_0}$ and $\gS(\g)^{\g}$. The quantum analogue $\widehat{\tilde\gZ}$ is generated by 
$\eus U(\g_0)^{\g_0}$ and $\eus U(\g)^{\g}$. 

For each irreducible finite-dimensional representation $V$ of $\g$, the restriction to $\g_0$ is 
multiplicity free. Hence the action of $\widehat{\tilde\gZ}$ on $V^{\gt n_0}$ has a simple spectrum. 
\end{ex}
  
\subsection{Classical applications}  \label{sec-cl}
Let us return to the Poisson side of the story. 

Suppose that there is a {\sf g.g.s.} for $\sigma$. Although $\tilde\gZ$ (or $\gZ$) is not a maximal 
Poisson-commutative subalgebra of $\gS(\g)$, it can be included into such a subalgebra in many natural ways.  
Let  $\mathcal C=\bbk[F_1,\ldots,F_{\bb(\g_0)}]$ be a maximal Poisson-commutative subalgebra of
$\gS(\g_0)$.
Then necessary $\gS(\gt g_0)^{\gt g_0}\subset \mathcal C$. 
 Suppose further that the $F_i$'s are homogeneous and their differentials 
are linearly independent on a big open subset of $\g_0^*$. For instance, one can take 
$\gc={\ca}_\gamma$ with $\gamma\in(\g_0)^*_{\sf reg}$, see~ \cite{mrl}. An easy calculation shows that 
$\mathsf{alg}\langle \tilde\gZ, \gc \rangle=\mathsf{alg}\langle\gZ, \gc\rangle$ has $\bb(\g)$ generators.
Indeed, $\tilde\gZ$ (or $\gZ$) has $\frac{1}{2}(\dim\g_1+\rk\g+\rk\g_0)$ free generators. Then we replace the generators sitting in $\gS(\g_0)$ (there are $\rk\g_0$ of them) with the whole bunch of generators of
$\gc$. In this way, we obtain
\[
  \frac{1}{2}(\dim\g_1+\rk\g+\rk\g_0)-\rk\g_0+\bb(\g_0)=\bb(\g)
\]
generators $\{F_i,  \boldsymbol{h}_j \mid 1\le i\le \bb(\gt g_0), 1\le j\le \bb(\gt g)-\bb(\gt g_0)\}$.
Furthermore,  the differentials $\{\textsl{d} F_i, \textsl{d}\boldsymbol{h}_j\}$
 are linearly independent at $x\in\gt g^*$ if and only if 
$\dim(\textsl{d}_x \tilde\gZ + \textsl{d}_x \mathcal C)=\bb(\gt g)$. 
Write $x=x_0+x_1$ with $x_i\in\gt g_i$ and suppose that $x_0\in(\gt g_0^*)_{\sf reg}$. 
Then 
\[
(\textsl{d}_x \tilde \gZ \cap \textsl{d}_x\mathcal C) \subset \gt g_0, \ \ 
\pi(x)(\gt g_0,\textsl{d}_x \tilde\gZ)=0, \ \text{ and hence } \ \ \textsl{d}_x \tilde\gZ \cap \textsl{d}_x\mathcal C=\gt g_0^{x_0}. 
\]
If in addition $\dim\textsl{d}_x \tilde\gZ=\trdeg \gZ$ and $\dim\textsl{d}_{x_0}\mathcal C=\bb(\gt g_0)$, 
then $\dim(\textsl{d}_x \tilde\gZ + \textsl{d}_x \mathcal C)=\bb(\gt g)$. 
In view of Theorem~\ref{thm:maxim2}{\sf (i)}, we can conclude that 
the differentials $\{\textsl{d} F_i, \textsl{d}\boldsymbol{h}_j\}$ are linearly independent 
on a big open subset of $\g^*$. 
Thus, Theorem~\ref{ppy-max} applies and assures that $\mathsf{alg}\langle \tilde\gZ, \gc \rangle$ 
is a maximal Poisson-commutative subalgebra of $\g$. 

Arguing inductively, one can produce a maximal Poisson-commutative subalgebra of $\gS(\g)$ 
from a chain of symmetric subalgebras
$$
\g=\g^{(0)}\supset \g^{(1)}\supset \g^{(2)} \supset\ldots \supset \g^{(m)},
$$
where $\g^{(m)}$ is Abelian and 
each symmetric pair $(\g^{(i)},\g^{(i{+}1)})$ has a {\sf g.g.s.}

\begin{rema}
{\sf (i)} For any simple Lie algebra $\g$, there is an involution $\sigma$ that has a {\sf g.g.s.}~\cite[Sect.\,6]{coadj}.  Therefore our construction of a maximal Poisson-commutative subalgebra of $\gS(\g)$ 
related to a chain of symmetric subalgebras works for any simple $\g$. 

{\sf (ii)}  In \cite[\S\,6]{v:sc}, limits of Mishchenko--Fomenko subalgebras were introduced. 
The  Poisson counterpart of the Gelfand--Tsetlin subalgebra of $\eus U(\gt{sl}_{n{+}1})$
related to the chain 
\[
\gt{sl}_{n+1}\supset\gt{gl}_n\supset\gt{gl}_{n-1}\supset \ldots \supset\gt{gl}_2\supset\gt{gl}_1 ,
\]
appears as one of these limit subalgebras, see also   Example~\ref{G-Ts-ex}. 
The key point of Vinberg's construction is that the Poincar\'e series of any limit subalgebra is the same as that of  ${\ca}_\gamma$ with $\gamma\in\gt g^*_{\sf reg}$.  
With a few exceptions, 
our approach produces Poisson-commutative subalgebras with 
different  Poincar\'e series.  This can be illustrated by the chain 
$$
\gt{so}_{5}\supset\gt{so}_4\supset \gt{so}_{2}\oplus\gt{so}_2 .
$$ 
Here the degrees of the generators of the related maximal Poisson-commutative subalgebra are 
$(4,2,2,2,1,1)$ opposite to $(4,3,2,1,2,1)$ in the case of ${\ca}_\gamma$. 
\end{rema}

Another feature is that  $\gZ$ can be used for constructing a 
Poisson-commutative subalgebra of $\gS(\g_0)$. 
Let $(\g,\g_0)$ be an arbitrary symmetric pair. If there is a {\sf g.g.s.} for $(\g,\g_0)$, then we are able 
to consider both algebras, $\gZ$ and $\tilde\gZ$.
For $\eta\in\g_1^*$, let  $\gZ_\eta$, $\tilde\gZ_\eta$ denote the restrictions of $\gZ$ and $\tilde\gZ$  to 
$\g_0^*+ \eta$. By choosing $\eta$ as the origin, we identify $\g^*_0+\eta$ with $\g_0^*$.
Then $\gZ_\eta$ and $\tilde\gZ_\eta$ are  homogeneous  subalgebras of $\gS(\g_0)$. 
Moreover, they Poisson-commute with $\g_0^\eta$. 

\begin{lm}       \label{apl0}
The subalgebras $\gZ_\eta$ and $\tilde\gZ_\eta$ are Poisson-commutative.
\end{lm}
\begin{proof}
Take $H,F\in\gZ$ or $H,F\in\tilde\gZ$ and $x\in\g_0^*$. 
Let ${\bf h}$ and ${\bf f}$ be the restrictions of $H,F$ to $\g_0^*+\eta$. 
Then  $\textsl{d}_{x+\eta} H=\textsl{d}_x{\bf h} + \xi_1$, 
$\textsl{d}_{x+\eta} F=\textsl{d}_x{\bf f} + \nu_1$, where $\xi_1,\nu_1\in\g_1$.
Set $\xi_0= \textsl{d}_x{\bf h}$, $\nu_0 = \textsl{d}_x{\bf f}$. Our goal is to show that 
$x([\xi_0,\nu_0])=0$. 

Since $H$ and $F$ commute w.r.t. any bracket $\{\,\,,\,\}_t$ with $t\in\mathbb P$, we have in particular 
$x([\xi_1,\nu_1])=0$, as well as $(x+\eta)([\xi_0+\xi_1,\nu_0+\nu_1])=0$. Both are also $\g_0$-invariants. Therefore 
$$
(x{+}\eta)([\xi_0,\nu_0+\nu_1])=0, \qquad 0=(x+\eta)([\nu_0,\xi_0+\xi_1])=x([\nu_0,\xi_0])+\eta([\nu_0,\xi_1]). 
$$
Now $0=(x+\eta)([\xi_1,\nu_0 +\nu_1])=\eta([\xi_1,\nu_0])$ and it is clear that $x([\xi_0,\nu_0])=0$.
\end{proof}

\begin{rmk}       \label{rem:man}
Let $(\g,\g_0)=(\gt{sl}_n,\gt{so}_n)$. The corresponding involution $\sigma$ is of
{\it maximal rank\/} and any set of generators $H_1,\ldots,H_l\in \gS(\g)^\g$ is a {\sf g.g.s.} for $\sigma$. 
The related Poisson-commutative subalgebra $\gZ$ appeared, in a way, in work
of Manakov~\cite{Man}. He stated that 
the restriction of $\gZ$ to $\g_0+\eta$ with $\eta\in\g_1$ is a Poisson-commutative 
subalgebra of $\gS(\g_0)$ of the maximal possible transcendence degree, which 
is $\boldsymbol{b}(\g_0)$.  Below we present a connection between his results and ours. 
We are grateful to E.B.\,Vinberg for bringing our attention to the fact that 
Manakov's construction involves an involution.  
\end{rmk}

Let $\ce_1\subset\g_1$ be a Cartan subspace. If $\eta\in\gt c_1$ is generic, then $\el:=\g_{0}^\eta$ is
reductive  
and it is also the centraliser of $\ce_1$ in $\g_0$. There are well-known equalities:
$\dim\g_1-\dim\g_0=\dim\gt l-\dim\ce_1$ and $\rk\el=\rk\g-\dim\ce_1$. 

\begin{thm}      \label{Man-in} 
For almost all $\eta\in\ce_1$, we have 
\begin{itemize}
\item[{\sf (i)}] \ $\trdeg \gZ_\eta = \boldsymbol{b}(\g_0)-\bb(\gt l)+\rk\gt l$;  
\item[{\sf (ii)}] \ if there is a {\sf g.g.s.} for $\sigma$, then $\tilde\gZ_\eta$ is a maximal Poisson-commutative 
subalgebra of $\gS(\g_0)^{\gt l}$. Besides, if  $\gt l$\/ is Abelian, then 
$\tilde\gZ_\eta$ is a maximal Poisson-commutative subalgebra of $\gS(\g_0)$. 
\end{itemize}
\end{thm}
\begin{proof}
Suppose that $\eta$ is generic enough. Then \\ \indent
\textbullet\quad $\dim\textsl{d}_y \gZ = \frac{1}{2}(\dim\g_1+\rk\g+\rk\g_0)$ 
for $y$ in a dense open subset of $\g_0+ \eta$, and 
\\ \indent
\textbullet \quad $\dim\textsl{d}_y \tilde\gZ = \frac{1}{2}(\dim\g_1+\rk\g+ \rk\g_0)$ for $y$ in a {\bf big} open subset of $\g_0+ \eta$. 
\\
Note that the subspaces $\textsl{d}_y \gZ$ and $\textsl{d}_y \tilde\gZ$
are orthogonal to $\g_0$ w.r.t. the bilinear form $\pi(y)=y([\,\,,\,])$. 
Hence for both of them, the intersection with $\g_1$ has dimension at most $\dim\gt c_1$. 
It is easily seen that actually $\dim(\textsl{d}_y \tilde\gZ \cap \g_1)=\dim\gt c_1$.
Furthermore, 
$$
\textsl{d}_y \gZ_\eta \simeq \textsl{d}_y\gZ /(\textsl{d}_y\gZ \cap \g_1)
$$
and the same  formula holds for $\tilde\gZ$.  
Therefore
\begin{multline*}
\trdeg \gZ_\eta \ge \frac{1}{2}(\dim\g_1+\rk\g+\rk\g_0)- \dim\gt c_1 = \frac{1}{2}
(\dim\g_1-\dim\gt c_1 +\rk\g - \dim\gt c_1+\rk\g_0)  \\
=\frac{1}{2} (\dim\g_0-\dim\gt l +\rk\gt l+\rk\g_0)= \boldsymbol{b}(\g_0)-\bb(\gt l)+\rk\gt l.
\end{multline*}
Since $\gZ_\eta\subset \gS(\g_0)^{\gt l}$ and $\rk\gt l=\ind\gt l$, the transcendence degree of 
$\gZ_\eta$ cannot be larger than 
$ \boldsymbol{b}(\g_0)-\bb(\gt l)+\rk\gt l$ by \cite[Prop.~1.1]{m-y}.
Because $\tilde\gZ$ in an algebraic extension of $\gZ$, we also have 
$\trdeg \tilde\gZ_\eta=\trdeg \gZ_\eta$. 

The difference $\trdeg \tilde\gZ -\trdeg\tilde\gZ_\eta$ is equal to $\dim\gt c_1$. 
We consider the algebra $\tilde\gZ$ only if there is a {\sf g.g.s} for $\sigma$. 
In that case the map $r_1$ is surjective and therefore for certain members
$H_i$ of the {\sf g.g.s.} we have $H_i^\bullet\in\gS(\g_1)$~\cite{coadj}.
The number of such element is equal to $\dim\ce_1$, and they restrict to constants on 
$\g_0^*+\eta$.  

We see  that $\tilde\gZ_\eta$ is freely generated by 
$\tilde H_1,\ldots,\tilde H_{\rk\g_0}\in\gS(\g_0)^{\g_0}$ and 
the restrictions to $\eta+ \g_0$ of $(H_j)_{(i,d_j-i)}$ with $0<i<d_j$. 
Moreover, the differentials of these generators are linearly independent on a big 
open subset. According to Theorem~\ref{ppy-max}, $\tilde\gZ_\eta$  is an algebraically closed subalgebra 
of $\gS(\g_0)$. By a standard argument, it is a maximal Poisson-commutative subalgebra
of $\gS(\g_0)^{\gt l}$. 

Suppose that $\gt l$ is Abelian. Then $\dim\gt l=\rk\gt l$ and 
$\tilde\gZ_\eta$ is a Poisson-commutative subalgebra of $\gS(\g_0)$ of the maximal possible 
transcendence degree. Here $\tilde\gZ_\eta$  is maximal in $\gS(\g_0)$. 
\end{proof}

The statements of Theorem~\ref{Man-in} are not entirely satisfactory. 
It would be nice to have an explicit description of $\eta$ such that  the results hold.
In the original setting of Manakov, $\gt l$ is trivial and the equality  
$\trdeg \gZ_\eta=\bb(\g_0)$ holds for each regular $\eta\in\gt c_1$, 
see \cite{Boj}.
But a more precise assertion requires a further analysis of $\g^*_{(t),\sf sing}$  and we prefer to postpone 
it.

\appendix   
\section{On pencils of skew-symmetric forms}  
\label{sect:app}
\setcounter{equation}{0}

\noindent 
Here we gather some general facts concerning  skew-symmetric bilinear forms.
Let $\eus P$ be a two-dimensional vector space of (possibly degenerate) skew-symmetric bilinear
forms on a finite-dimensional vector space $V$. Set $m=\max_{A\in \eus P }\rk A$, and let 
$\eus P_{\sf reg}\subset \eus P$ be the set of all forms of rank $m$. Then $\eus P_{\sf reg}$ is a conical
open subset of $\eus P$.
For each $A\in \eus P$, let $\ker A\subset V$ be the kernel of $A$. 
Our object of interest is the subspace 
$L:=\sum_{A\in \eus P_{\sf reg}} \ker A$.  

\begin{lm}[{\cite[Appendix]{mrl}}]              \label{open}
If\/ $\Omega$ is a non-empty open subset of $\eus P_{\sf reg}$,
then $\sum_{A\in \Omega} \ker A=L$.
\end{lm}

\begin{cl}          \label{cl-ort}
For  all $A, B \in \eus P\setminus\{0\}$, we have   $A(\ker B,L)=0$ and therefore $A(L,L)=0$. 
\end{cl}
\begin{proof}
Clearly, the equality $A(\ker B,L)=0$ holds if $B$ is a scalar multiple of $A$. If not, then
we consider $L_{b}:=\ker(A+b B)$ for $b\in\bbk$. Here
\[
     A(\ker B,L_b)=(A+bB)(\ker B,L_b)-bB(\ker B,L_b)=0.
\]
By Lemma~\ref{open},  there is an open subset $\mathfrak O\subset\bbk$ such that
$L$ is spanned by $\{L_{b}\mid b\in \mathfrak O\}$. Hence
\[
    A(\ker B,L)= A(\ker B,\textstyle \sum_{b\in \mathfrak O}L_b)=0 .   \qedhere
\] 
\end{proof} 

Suppose that $C\in\eus P\setminus \eus P_{\sf reg}$. Then $U=\ker C$ may not be a subspace 
of $L$.  Take $A\in\eus P\setminus\{0\}$ that is not proportional to $C$ and restrict it to $U$. 
The resulting skew-symmetric form on $U$ does not change if we replace $A$ with any $A\,+ \,b\,C$, where  $b\in\bbk$.

\begin{lm}         \label{restr}
Let $C$, $A$, and $U$ be as above. Then $\rk(A|_{U})\le \dim U-(\dim V-m)$.
\end{lm}
\begin{proof}
By Corollary~\ref{cl-ort}, we have $A(U,L)=0$. Set $r=\dim V-m$. 
Because  $\eus P$ is irreducible, $\overline{\eus P_{\sf reg}}=\eus P$ and 
there is a curve $\tau\!:\bbk^\times \to \eus P_{\sf reg}$
such that $\lim_{t\to 0} \tau(t) = C$. Hence 
\[
\lim_{t\to 0}(\ker\tau(t)) \subset \ker C,
\]
where the limit is taken in the Grassmannian of the $r$-dimensional subspaces of $V$.  
Set $U_0:=\lim_{t\to 0}(\ker\tau(t))$.
If $t\ne 0$, then 
$\ker\tau(t)\subset L$  and $A(\ker\tau(t),U)=0$. Hence also 
$A(U_0,U)=0$ and $U_0\subset \ker (A|_U)$. 
It remains to notice that $\dim U_0=r$. 
\end{proof}

\noindent 
{\bf Remark.} Lemma~\ref{restr} implies Vinberg's inequality: if $\q$ is Lie algebra, then $\ind\gt q^\gamma\ge \ind\gt q$ for 
any $\gamma\in\gt q^*$, see~\cite[Cor.\,1.7]{p03}. 
 
\begin{thm}\label{sum-dim}
Suppose that $\eus P\setminus \eus P_{\sf reg}= \bbk C$ with $C\ne 0$ and $U=\ker C$. Keep the notation of 
Lemma~\ref{restr} and suppose further that $\rk(A|_U)=\dim U {-}\dim V+ m$.
Then $\dim (L\cap U)=\dim V-m$ and $\dim L = (\dim V-m)+\frac{1}{2}(\dim V - \dim U)$.
\end{thm} 
\begin{proof}
Let $B\in\eus P_{\sf reg}$ be non-proportional to $A$.  Given $A,B\in \eus P_{\sf reg}$,
there is the so-called {\it Jordan--Kronecker canonical form\/} of $A$ and $B$, see \cite{JK}.  
Namely, $V=V_1\oplus\dots\oplus V_d$, where $A(V_i,V_j)=0=B(V_i,V_j)$ for $i\ne j$, and 
accordingly, $A=\sum A_i$ and $B=\sum B_i$. There are two possibilities for 
$(A_i,B_i)$, one obtains either a {\it Kronecker} or a {\it Jordan block\/} here, see figures below. 
Assume that $\dim V_i>0$ for each $i$. 

\vskip-10pt
\[
\hskip-30pt
\begin{array}{lcc}
  &  A_i  &  B_i     \\  & & \\
\begin{array}{c} \text{A Jordan block } \\ (\lambda_i\in \bbk) \end{array}: &
\begin{pmatrix}
     &   \!\!\! \eus J(\lambda_i) \\  
   \!  -\eus J^\top(\lambda_i) &
 \end{pmatrix} &
 \begin{pmatrix}
   &  -\mathrm{I}  \\  
   \mathrm{I} &
 \end{pmatrix} , \\     & &
 \\

\begin{array}{c} \text{a Kronecker} \\
\text{block} \end{array}: &
 \!\!\!\!\!\!\!\!\!\!\!\! \begin{pmatrix}
     &
\hskip-10pt \boxed{ \begin{matrix}
  1 & 0  & & \\
      &\ddots & \ddots & \\
     &           &  1 & 0
     \end{matrix}}
     \\
\boxed{\begin{matrix}
\!\!\! -1  &             &   \\
  0 & \ddots &   \\
     & \ddots &\!\!\! -1 \\
     &             &0
  \end{matrix}
  }\end{pmatrix}
 &
  \begin{pmatrix}
     &
\hskip-10pt \boxed{ \begin{matrix}
  0 & 1  & & \\
      &\ddots & \ddots & \\
     &           &  0 & 1
     \end{matrix}}
     \\
\boxed{\begin{matrix}
 0  &             &   \\
  \!\!\! -1 & \ddots &   \\
     & \ddots &0 \\
     &             &\!\!\! -1
  \end{matrix}
  }\end{pmatrix},
\end{array}
\]
where $\eus J(\lambda_i)=\begin{pmatrix}
\lambda_i & 1 & &\\
& \lambda_i & \ddots & \\
& & \ddots & 1 \\ & & & \lambda_i
\end{pmatrix}.
$   In general, there can occur ``Jordan blocks with $\lb_i=\infty$'', but this is not the case here, since $B\in\eus P$ is assumed to be regular.

Note that if $V_i$ gives rise to a Jordan block, then $\dim V_i$ is even and  
both $A_i$ and $B_i$ are non-degenerate on $V_i$. For a Kronecker block, 
$\dim V_i=2k_i+1$, $\rk A_i=2k_i=\rk B_i$ and the same holds for every non-zero linear combination 
of $A_i$ and $B_i$. 

There is a unique $\lb\in\bbk\setminus\{0\}$ such that $C=A+\lb B$. This $\lb$ can be determined as 
the root of the equation $\det(A_i+\lb B_i)=0$ for any Jordan block $(A_i,B_i)$. This readily follows from  
the uniqueness of the singular line $\bbk C\subset \eus P$. On the other hand, the above matrices show 
that the root corresponding to $(A_i,B_i)$ is $\lb_i$.
Therefore, all $\lb_i$'s are equal and
coincide with $\lb$.  
 
Let us assume that $V_i$ defines a Kronecker block if and only if $1\le i\le d'$. Then necessarily 
$d'=\dim V-m$. Let $\ker(A_i+bB_i)\subset V_i$ be the kernel of the bilinear form $A_i+bB_i$. 
Then 
$$
   L=\bigoplus_{i=1}^{d'}\sum_{b:\,A+bB\in\eus P_{\sf reg}} \ker(A_i+bB_i)=: \bigoplus_{i=1}^{d'} L_i.
$$ 
It follows from the above matrix form of a Kronecker block that $\dim L_i=k_i+ 1$, cf. also \cite[Appendix]{mrl}.   

Set $C_i=A_i+\lambda B_i$ for each $i\in \{1,2,\dots,d\}$. It is a bilinear form on $V_i$.

\textbullet \quad
If $i\le d'$, then  $\dim\ker C_i=1$.  Therefore $\ker C_i\subset L_i$ and $\dim(\ker C\cap L)=d'$. 

\textbullet \quad
If $i>d'$,  then $\dim\ker C_i=2$. 
\\
Hence $\dim U = 2(d-d')+d'=2d-d'$. 
Since $U=\bigoplus_{i=1}^d \ker C_i$ and the spaces $\{\ker C_i\}$ are pairwise orthogonal w.r.t. any
form in $\eus P$, we have $A(\ker C_j,U)=0$ for $j\le d'$. Hence the condition 
$\rk(A|_U)=\dim U-\dim V+ m$ implies that $A_i$ is non-degenerate on 
$\ker C_i$ for any $i>d'$. The explicit matrix form of a Jordan block shows that $\ker C_i$ is 
spanned by two middle basis vectors of $V_i$. Therefore, $A_i$ is non-degenerate on 
$\ker C_i$ if and only if  $\dim V_i=2$, and hence $C_i=0$.  

Summing up, we obtain
$$
  \dim L=\sum_{i=1}^{d'} (k_i+1) = d'+ \sum_{i=1}^{d'} \frac{1}{2}\rk C_i=d'+\frac{1}{2}\rk C=(\dim V{-}m)+\frac{1}{2}(\dim V{-}\dim U).
$$ 
This completes the proof. 
\end{proof}


\begin{thebibliography}{Pa95a}

\bibitem[B91]{bols}
{\sc A.~Bolsinov}. Commutative families of functions related to consistent Poisson brackets, 
{\it Acta Appl. Math.}, {\bf 24}, no.\,3 (1991), 253--274.

\bibitem[BB02]{bol-bor}
{\sc A.~Bolsinov} and {\sc A.~Borisov}.
Compatible Poisson brackets on Lie algebras. (Russian) {\it Mat. Zametki} {\bf 72}\,(2002), no. 1, 11--34; translation in {\it Math. Notes} {\bf 72}\,(2002), no. 1-2, 10--30.

\bibitem[BK79]{bokr}
{\sc W.~Borho} and {\sc H.~Kraft}. 
\"Uber Bahnen und deren Deformationen bei linearen Aktionen reduktiver Gruppen,
{\it Comment. Math. Helv.}  {\bf 54}\,(1979), no.\,1, 61--104.

\bibitem[DZ05]{duzu}
{\sc J.-P.~Dufour} and  {\sc N.T.~Zung}.  ``Poisson structures and their normal forms''. Progress in Mathematics, {\bf 242}. Birkh\"auser Verlag, Basel, 2005.

\bibitem[FFR]{FFR}
{\sc B.~Feigin}, {\sc E.~Frenkel}, and {\sc L.~Rybnikov}.
Opers with irregular singularity and spectra of the shift of argument subalgebra,
{\it Duke Math. J.}, {\bf 155}\,(2010), no.\,2, 337--363.

\bibitem[GDI]{Boj}
{\sc B.~Gaji\'c, V.~Dragovi\'c,} and  {\sc B.~Jovanovi\'c}.
On the completeness of the Manakov integrals,
{\it J. Math. Sci.}, {\bf 223}:6 (2017), 675--685.

\bibitem[GT50]{gt-1}
{\sc I.M.~Gelfand} and {M.L.~Tsetlin}. 
Finite-dimensional representations of the group of unimodular matrices. (Russian) 
{\it Doklady Akad. Nauk SSSR} (N.S.) {\bf 71}\,(1950), 825--828.
English transl. in: {\sc I.M.~Gelfand}, Collected Papers, vol. II, Springer-Verlag, Berlin, 1988, pp. 653--656.

\bibitem[GT50']{gt-2}
{\sc I.M.~Gelfand} and {M.L.~Tsetlin}. 
Finite-dimensional representations of groups of orthogonal matrices. (Russian) 
{\it Doklady Akad. Nauk SSSR} (N.S.) {\bf 71}\,(1950), 1017--1020. English transl. in:
{I.M.~Gelfand}, Collected Papers, vol. II, Springer-Verlag, Berlin, 1988, pp. 657--661.

\bibitem[K63]{ko63}
{\sc B.\,Kostant}. Lie group representations on
polynomial rings, {\it  Amer. J. Math.}, {\bf 85}\,(1963), 327--404.

\bibitem[M76]{Man}
{\sc S.V.\,Manakov}.
 Note on the integration of Euler's equations of the dynamics of an $n$-dimensional rigid body,
{\it Funct. Anal. Appl.}, {\bf 10}\,(1976), no.4, 93--94 (in Russian).

\bibitem[MF78]{mf}
{\sc A.S.\,Mishchenko} and {\sc A.T.\,Fomenko}.
Euler equation on finite-dimensional Lie groups,
{\it Math. USSR-Izv.} {\bf 12}\,(1978), 371--389.

\bibitem[M07]{book}
{\sc A. Molev}, 
``Yangians and Classical Lie Algebras''. Mathematical Surveys and Monographs, {\bf 143},  
American Mathematical Society, Providence, RI, 2007.

\bibitem[MO00]{mo}
{\sc A.~Molev} and {\sc G.~Olshanski}.
Centralizer construction for twisted Yangians, 
{\it Selecta Math.}, {\bf 6}\,(2000), no.\,3, 269--317.  

\bibitem[MY]{m-y}
{\sc A.~Molev} and {\sc O.~Yakimova}. 
Quantisation and nilpotent limits of Mishchenko--Fomenko subalgebras, 
(\href{http://arxiv.org/abs/1711.03917}{\tt \bfseries arxiv:1711.03917v1 [math.RT]}, 32 pp.)

\bibitem[O91]{o}
{\sc G.I.~Olshanski}. 
Representations of infinite-dimensional classical groups, limits of enveloping algebras, and 
Yangians, In: {\it Topics in Representation Theory}, Advances in Soviet Math. (A.A. Kirillov ed.), 
vol. {\bf 2}, AMS, Providence   RI, 1991, 1--66.  
	
\bibitem[P03]{p03} 
{\sc D.~Panyushev}. 
The index of a Lie algebra, the centraliser of a nilpotent element,
and the normaliser of the centraliser, {\it Math. Proc. Camb. Phil. Soc.},
{\bf 134}, Part\,1 (2003), 41--59.

\bibitem[P07]{p05} 
{\sc D.~Panyushev}. Semi-direct products of Lie algebras and their invariants,
{\it Publ. RIMS}, {\bf 43}, no.\,4 (2007), 1199--1257.

\bibitem[P07']{coadj} 
{\sc D.~Panyushev}.
On the coadjoint representation of $\mathbb Z_2$-contractions of reductive Lie algebras,
{\it Adv. Math.}, {\bf 213}\,(2007), 380--404. 

\bibitem[PPY]{trio} 
{\sc D.~Panyushev}, {\sc A.~Premet} and {\sc O.~Yakimova}. 
On symmetric invariants of centralisers in reductive  Lie algebras, 
{\it J. Algebra}, {\bf 313}\,(2007), 343--391.

\bibitem[PY08]{mrl} {\sc D.~Panyushev} and {\sc O.~Yakimova}.  
The argument shift method and maximal commutative subalgebras of Poisson algebras, 
{\it Math. Res. Letters}, {\bf 15}, no.\,2 (2008), 239--249.

\bibitem[PY12]{alafe} {\sc D.~Panyushev} and {\sc O.~Yakimova}. 
On a remarkable contraction of semisimple Lie algebras,  
{\it Annales Inst. Fourier} (Grenoble), {\bf 62}, no.\,6 (2012), 2053--2068.

\bibitem[R06]{r:si}
{\sc L. G. Rybnikov},
{The shift of invariants method and the Gaudin model},
{\it Funct. Anal. Appl.} {\bf 40}\,(2006), 188--199.

\bibitem[Sl80]{815} 
{\sc P.~Slodowy}. 
"Simple singularities and simple algebraic groups",
Lect. Notes Math. {\bf 815},  Berlin: Springer, 1980.

\bibitem[S74]{springer} 
{\sc T.A.~Springer}.
Regular elements of finite reflection groups, {\it  Invent. Math.}
{\bf 25}\,(1974), 159--198.

\bibitem[T91]{JK}
{\sc R.C.\,Thompson}. 
Pencils of complex and real symmetric and skew matrices, 
{\it Linear Algebra and its Appl.},  {\bf 147}\,(1991), 323--371. 

\bibitem[V68]{va} 
{\sc V.S.~Varadarajan}. 
On the ring of invariant polynomials on a semisimple Lie algebra, 
{\it Amer. J. Math.}, {\bf 90}\,(1968), 308--317. 

\bibitem[Vi91]{v:sc}
{\sc E.B.~Vinberg}.
{Some commutative subalgebras of a universal enveloping algebra},
{\it Math. USSR-Izv.} {\bf 36} (1991), 1--22.

\bibitem[VP89]{VP}
{\rusc {E1}.B.~Vinberg, V.L.~Popov}. {\rusi ``Teoriya Invariantov"}, {\rus V:  
Sovrem{.} probl{.} matematiki. Fundamental{\cprime}nye napravl., t.\,55,
str.}\,137--309. {\rus Moskva: VINITI} 1989 (Russian).
English translation: 
{\sc V.L.~Popov} and {\sc E.B.~Vinberg}. ``Invariant theory", In: {\it  Algebraic Geometry IV}
(Encyclopaedia Math. Sci., vol.~55, pp.123--284) 
Berlin Heidelberg New York: Springer 1994.

\bibitem[Y14]{contr}
{\sc O.~Yakimova}.  
One-parameter contractions of Lie-Poisson brackets, 
{\it J. Eur. Math. Soc.}, {\bf 16}\,(2014), 387--407.

\bibitem[Y17]{Y-imrn}
{\sc O.~Yakimova}. 
Symmetric invariants of $\mathbb Z_2$-contractions and other semi-direct products, 
{\it Int.  Math. Res. Notices}, (2017) 2017 (6): 1674--1716. 

\end{thebibliography}
\end{document}